\documentclass{article}
\usepackage{relsize}

\usepackage[dvipsnames]{xcolor}

\usepackage[english]{babel}

\usepackage[letterpaper]{geometry}

\usepackage{rotating}

\usepackage{amsfonts}
\usepackage{amsmath}
\usepackage{amsthm}
\usepackage{graphicx}
\usepackage[colorlinks=true, allcolors=blue]{hyperref}
\usepackage{captdef}
\usepackage{caption}
\usepackage{comment}
\usepackage{hyperref}

\usepackage{tikz}

\newcommand{\R}{\mathbb{R}} 
\newcommand{\C}{\mathbb{C}} 
\newcommand{\Q}{\mathbb{Q}} 
\newcommand{\Z}{\mathbb{Z}} 

\usepackage{enumitem}
\usepackage{pifont}

\usepackage{array}

\newtheorem{defi}{Definition}[section]
\newtheorem{prop}{Proposition}[section]
\newtheorem{lemma}{Lemma}[section] 
\newtheorem{thm}{Theorem}[section]
\newtheorem{corr}{Corollary}[section]
 
\newtheorem{rmk}{Remark}

\newtheorem*{thmN}{Theorem}

\title{Outer contact billiards}

\author{Ana Chavez-Caliz\footnote{
Instituto de Matem\'aticas, 
Universidad Nacional Aut\'onoma de M\'exico,
62210 Cuernavaca, 
Mexico; 
ana.chavez@im.unam.mx} 
\and 
Connor Jackman\footnote{
Department of Mathematics, 
Instituto Tecnol\'ogico Aut\'onomo de M\'exico, 
01080 Mexico City,
Mexico; 
connor.jackman@itam.mx}}
\date{\today}

\begin{document}

\maketitle
\begin{abstract}
    We introduce outer contact billiards as an odd dimensional counterpart to outer symplectic billiards. 
    Until now, outer billiards have only been considered in even dimensional symplectic vector spaces. By projectivizing outer symplectic billiards we obtain outer contact billiards, where the affine midpoint condition descends to its projective analog, namely harmonic conjugation. 

   We prove that outer contact billiards generates contactomorphisms.  
    For quadratic surfaces in $\mathbb{RP}^3$, we show that the 
    correspondence is completely integrable: its domain is foliated by invariant quadrics 
    and, on each leaf, the dynamics is determined by the iteration of an explicit linear 
    transformation. We also establish two rigidity results for periodic trajectories: 
    outer contact billiards admit no $3$-periodic orbits and, among quadratic tables, only 
    one admits $4$-periodic orbits.
    
\end{abstract}

\tableofcontents

\section{Introduction}

In 1959, B. Neumann (\cite{neumann1959sharing}) introduced the outer billiard map, a 
dynamical system defined outside a convex planar domain. Unlike the classical Birkhoff 
billiard, where the trajectory evolves inside the table by reflecting at its boundary, 
the outer billiard map acts on the exterior of the table: a point is sent to its 
reflection across a tangency point of the table (see Figure \ref{fig:outerBilliard}). 
Motivated by questions concerning the stability of the solar system, J. Moser continued 
the study of the subject in the 1970s (\cite{Mo1}, \cite{Mo2}), initiating a long line of research that continues to this day. 

In \cite{tabachnikov1995dual}, S. Tabachnikov introduced outer symplectic billiards, 
extending the construction from the plane to arbitrary linear symplectic spaces. In this 
setting, the characteristic direction determined by the symplectic structure replaces the 
tangent line of the planar construction, yielding a natural higher-dimensional analog of 
outer billiards. This topic has been further studied in \cite{tabachnikov2003three}, 
\cite{albers2024outer}, and \cite{albers2026outer}, among others. 

\begin{center}
    \includegraphics[width=0.4\linewidth]{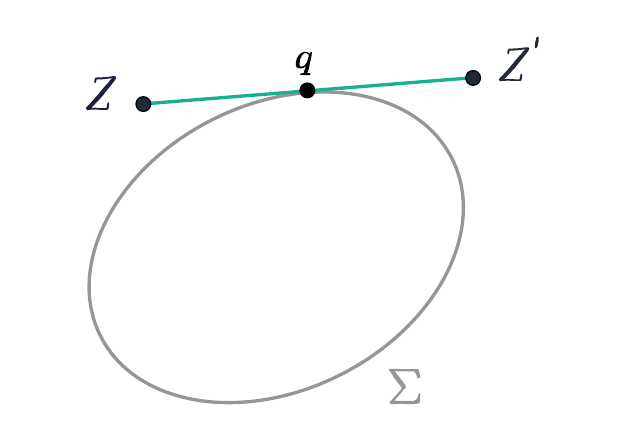}
    \figcaption{The outer billiard map in the plane.}
    \label{fig:outerBilliard}
\end{center}
Since symplectic manifolds are necessarily even-dimensional, the present paper addresses 
the following natural question: {\em what should outer billiards look like in odd 
dimensions?}\\

Following the Arnold-Givental philosophy: that contact geometry is to symplectic geometry as projective geometry is to affine geometry (see \cite{EDSIV}, \cite{tabachnikov1993geometry}), we propose an answer to this question. This analogy suggests that the outer symplectic billiards should admit a projective counterpart. Starting with a symplectic vector space $(V, \omega)$, consider its projectivization $\mathbb{P}V$. Its projectivization carries a canonical contact structure: the contact hyperplane at a given line through the origin is its symplectic complement. From this perspective, the passage from outer symplectic billiards to outer contact billiards amounts to replacing the affine ingredient of the construction, the midpoint condition, with its projective counterpart, namely harmonic conjugation. In other words, outer contact billiards are a scaling reduction of outer symplectic billiards (see, for example, \cite{ContRed}). This leads to a correspondence in projective space, which we call the {\em outer contact billiard correspondence}: given a hypersurface $\Sigma$, a pair of points, $Z,Z'\in\mathbb{P}V$, are in correspondence over $\Sigma$ if, first of all, they lie on some characteristic line of $\Sigma$. For instance, in $\mathbb{RP}^3$ the characteristic lines of a surface are the lines of intersection between the tangent and contact planes as the points run over the surface. Furthermore, each characteristic line of $\Sigma$ contains two distinguished points: $q,q'$ (see Definition \ref{def:CharLine}) and we ask the points $Z,Z'$, to be harmonic conjugates with respect to these distinguished points $q, q'$ (see Definition 
\ref{def:OutContBill}, or Figure \ref{fig:contactcorrespondence} below).


Here we will describe this correspondence explicitly when $\Sigma\subset\mathbb{RP}^3$ is a quadratic surface. Also we show in general 
(see \S \ref{sec:GenThms} below for more precise statements) that:
\begin{thmN}
    The outer contact billiard correspondence generates local contactomorphisms. Moreover, it is the only choice of outer billiard generating local contactomorphisms over any table.
\end{thmN}

Besides its simple geometric nature, this construction also has a natural dynamical 
motivation. A hypersurface in a contact manifold inherits a certain characteristic line field, i.e.,~first-order differential equation, whose integral curves play a central role in 
contact geometry and appear naturally, for example, in describing blow-ups of 
singularities in mechanical systems (see e.g.~\cite{SloanDS}, \cite{ContRed}). The outer contact billiard correspondence 
provides discrete dynamics built along these same characteristic lines. In other words, contact billiards is a certain type of contact integrator (see e.g.~\cite{DiscHerg}, \cite{ContInt}) well adapted to the projective geometry of our situation. One expects contact billiard orbits to capture qualitative features of the corresponding integral 
curves, a behavior which is already apparent in the quadratic examples considered in 
Section~\ref{sec:Exs}.\\

The paper is organized as follows. In Section~\ref{sec:defi}, we recall the projective 
geometric notions needed to define the correspondence. Section~\ref{sec:Exs} contains some 
explicit examples associated with quadratic hypersurfaces, illustrating both the 
similarities and the differences between outer contact and outer symplectic billiards. 
We show that every quadratic example is completely integrable: the domain is filled by 
invariant leaves, and on each one, the billiard map is described by a fixed linear 
transformation. Appendix~\ref{sec:Quads} shows that, up to a change of symplectic 
coordinates, our list of examples covers all quadratic tables in $\mathbb{RP}^3$. In 
Section~\ref{sec:periodic}, we study periodic orbits; in particular, we prove that the 
correspondence admits no three-periodic orbits (Proposition \ref{prop:No3RP3}), and that 
there is only one quadratic table admitting a $4$-periodic orbit in $\mathbb{RP}^3$ 
(Proposition \ref{prop:4PerQuad}). Section~\ref{sec:GenThms} establishes that the 
correspondence preserves the ambient contact structure (Theorem \ref{thm:ContTransf}). 
Although projective geometry has no intrinsic notion of midpoint, two natural settings do: 
affine charts and a spherical chart. We therefore ask for which hypersurfaces the harmonic 
conjugate construction agrees with the corresponding midpoint construction used in outer 
symplectic billiards in these models. Theorems \ref{thm:AffTable} and \ref{Thm:CplxTable} 
provide complete characterizations in the affine and spherical settings, respectively. 
Finally, Section~\ref{sec:questions} collects several open problems. We hope that the 
present work serves as a first step toward a broader theory of outer contact billiards.\\

\textbf{Acknowledgments.} ACC was supported by the DFG through Project-ID 281071066 – TRR 
191, and by the Simons Foundation International through the grant SFI-MPS-T-Institutes-
00011977 JS.

CJ was supported by the DFG through Project-ID 281071066 – TRR 191, and by the Asociación 
Mexicana de Cultura, A.C. \\

We thank Peter Albers, Sergei Tabachnikov, Lina Deschamps, and Levin Maier for their valuable 
insights. We are also grateful to Alessandro Bravetti, Gil Bor, Richard Montgomery, Sergio Zamora Barrera, and Orsola 
Capovilla-Searle for their generous and fruitful discussions.


\section{Outer contact billiards}
\label{sec:defi}

We first fix notation and recall some standard notions from projective (see e.g.~\cite{pottmann2001computational}), and contact geometry (see e.g.~\cite{EDSIV}).\\

Let $V$ be a vector space. Its projective space $\mathbb{P}V$ is the space of lines 
through the origin of $V$. We denote projective points by $Z \in \mathbb{P}V$ and write 
$\boldsymbol{Z} \in Z \setminus \{0\}\subset V$ for a choice of lift. \\

If $A,B,C,D \in \mathbb{P}V$ are collinear, one can always choose lifts 
$\boldsymbol{A}, \boldsymbol{B}, \boldsymbol{C}, \boldsymbol{D} \in V$ so that 
$\boldsymbol{C}= \boldsymbol{A}+\boldsymbol{B}$. Such lifts are unique up to simultaneous 
rescaling. If $\boldsymbol{D}=\lambda\boldsymbol{A}+\eta\boldsymbol{B}$ then the quotient
\[\text{cr}(A,B,C,D):=\dfrac{\lambda}{\eta},\]
is independent of the chosen lifts and its called the \emph{cross ratio} of $A, B, C, D$. In some affine coordinate along the line containing the quadruple of collinear points it is given by the usual formula $\text{cr}(A,B,C,D) = \frac{(A-C)(B-D)}{(A-D)(B-C)}$.

The points $C \text{ and } D$ are said to be \emph{harmonic conjugates} with respect to
$A \text{ and } B$ when $\text{cr}(A,B, C, D)=-1$, or equivalently, if we can write 
$\boldsymbol{C}= \boldsymbol{A}+\boldsymbol{B}$ and $\boldsymbol{D}=\boldsymbol{A}-\boldsymbol{B}$ (see Figure \ref{fig:HarmonicConjugates}). \\

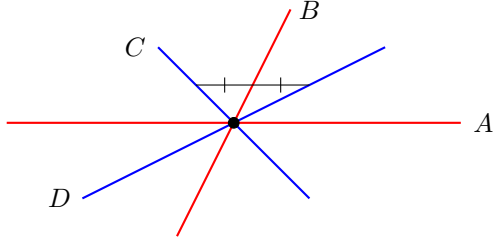
\begin{figure}[h]
\centering
\begin{tikzpicture}  
    
    \draw[red,thick] (-3,0) -- (3,0); 
    \draw[red,thick] (-.75,-1.5) -- (.75,1.5); 
    
    \draw[blue,thick] (-1,1) -- (1, -1); 
    \draw[blue,thick] (-2,-1) -- (2, 1); 

    \filldraw[black] (0,0) circle (2pt);

    \draw (-.5,.5) -- (1,.5);
    \draw (-.12,.6) -- (-.12,.4);
    \draw (.62,.6) -- (.62,.4);

    \node at (3.3,0) {$A$};
    \node at (1,1.5) {$B$};
    \node at (-1.3,1) {$C$};
    \node at (-2.3,-1) {$D$};
    
\end{tikzpicture}

    \caption{A harmonic tuple: $C,D$ are harmonic conjugates with respect to $A,B$ (and 
    vice-versa).}
    \label{fig:HarmonicConjugates}
\end{figure}

Now, suppose the given vector space is also equipped with a symplectic form $\omega$. 
The projective space inherits a contact structure $\xi$, assigning to a line through the 
origin of $V$ its symplectic orthogonal complement (a hyperplane containing said line), 
that is:
\begin{defi}\label{def:Polarity}
    The {\em contact plane} through $q\in \mathbb{P}V$ is the plane
\[ \xi_q = q^\omega := \{ {\bf u}\in V : \omega({\bf u},{\bf q}) = 0\} \in \mathbb{P}V^*.\]
Likewise, given a hyperplane $\Pi \in \mathbb{P}V^*$, its \emph{contact point} is the 
unique point 
$q_\Pi = \Pi^\omega  \in \Pi$ satisfying \[\Pi= \xi_{q_\Pi}.\]
A {\em Legendrian line} in $(\mathbb{P}V,\xi)$ is a line contained in some $\xi_q$ and 
passing through $q$.
\end{defi}

Throughout the paper, we will often work with a specific situation: 
\begin{defi}\label{def:CharLine}
    If $\Sigma \subset \mathbb{P}V$ is a regular hypersurface and $q\in \Sigma$, we write 
$q' = (T_q\Sigma)^\omega$ for the contact point of the tangent hyperplane $T_q\Sigma$. When $q'\ne q$,
the \emph{characteristic line} at $q\in\Sigma$ is the line joining $q$ and $q'$. 
\end{defi}

In computations, we may often use the following observation (immediate from the 
definitions)
\begin{prop}\label{prop:SGrad}
    For a hypersurface given (locally) as $\Sigma = f^{-1}(0)\subset \mathbb{P}V$ for 
    $f:V\dashrightarrow \R$ some homogeneous function, then $q' = \text{span}\{ X_f(q)\}$ 
    for $X_f$ the symplectic gradient of $f$: $\omega(\cdot , X_f) = df(\cdot)$.
\end{prop}

We are now ready to introduce the outer contact billiard correspondence (see Figure 
\ref{fig:contactReduction}). \\

\begin{defi}\label{def:OutContBill}
    Let $\Sigma\subset (\mathbb{P}V,\xi)$ be a hypersurface (the {\em table}). A 
    pair of points $Z,Z'\in \mathbb{P}V$ are in {\em outer contact billiards
    correspondence} over $\Sigma$ if there exist $q\in \Sigma$ such that
    \begin{enumerate}
        \item The line $Z\wedge Z'$ is the characteristic line at $q$, and 
        \item $Z,Z'$ are harmonic conjugates with respect to $q, q'$.
    \end{enumerate}
\end{defi}

Since outer contact billiards are, in general, a correspondence rather than a map, we will call broadly:
\begin{defi}\label{def:periodic}
    An {\em orbit} is a sequence $(Z_j)_{j\in \Z}$ such that each consecutive pair is 
    in correspondence. An orbit is {\em periodic} if there exists a positive integer 
    $k$ such that $Z_{j+k}=Z_j$ for all $j$. The smallest such $k$ is called the 
    {\em period} of the orbit.
\end{defi}
For instance, we will not worry here about when there is some orientation convention which could direct orbits in some standard way, and, unless mentioned otherwise, will be interested in non-collinear orbits.

\begin{rmk}
    A point $q\in \Sigma$ satisfying $T_q\Sigma = \xi_q$ is called a \emph{singular 
    point}. At such a point, the contact point of the tangent hyperplane coincides with 
    $q$ itself, that is, $q=q'$. Consequently, the characteristic line is no longer 
    determined, and a further analysis (blowup) of the contact billiards relation around such 
    points is necessary.
    
    Since contact structures are maximally non-integrable, one expects the singular set 
    to be small; indeed, for a generic hypersurface, the singular points are isolated, 
    and the set of singular points always has an empty interior. On the other hand, 
    topological obstructions may force their existence. For instance, if $\Sigma$ is a 
    2-sphere, then by the hairy ball theorem, the characteristic line field cannot be 
    everywhere non-vanishing, and therefore $\Sigma$ necessarily contains singular 
    points.
\end{rmk}

\begin{rmk}
    The contact structure $(\mathbb{P}V,\xi)$ is the model example of a contact 
    reduction of $(V\backslash 0,\omega)$ with respect to the radial scaling symmetry 
    \cite{ContRed}. Our definition of outer contact billiards is, in this sense, a scaling 
    reduction of symplectic billiards along the cone spanned by $\Sigma$ in $V$. See 
    Figure \ref{fig:contactReduction}.
\end{rmk}
\begin{rmk}
    Outer contact billiards depend only on the conformal class of the linear symplectic
    structure $(V,\omega)$.
\end{rmk}

\begin{figure}[h]
    \centering
    \includegraphics[width=0.45\linewidth]{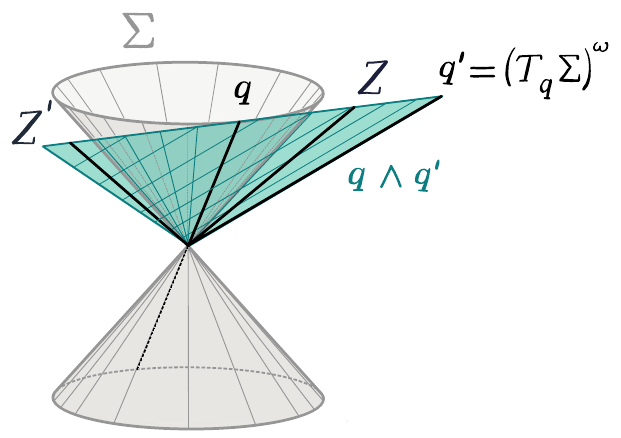}
    \caption{The outer contact billiard correspondence is obtained by projectivizing the 
    outer symplectic billiard correspondence on the cone over the table. The contact point 
    $q' = (T_q\Sigma)^\omega \in \mathbb{P}V$, and $q\in \Sigma$ generate the 
    characteristic line $q\wedge q'$, along with the involution of $q\wedge q'$ by 
    taking harmonic conjugates, $Z\longleftrightarrow Z'$, with respect to $q, q'$ (as 
    shown in Figure \ref{fig:HarmonicConjugates}).}
    \label{fig:contactReduction}
\end{figure}

\begin{rmk}
    We also propose (\S \ref{sec:questions} below) an analogous definition for an \emph{inner contact billiards} in $\mathbb{RP}^3$, whose study we will leave for future work.
\end{rmk}

We often consider the linear symplectic structure on $V$ as an identification
\[ \omega:V \to V^*, ~~\boldsymbol{v}\mapsto \iota_{\boldsymbol{ v}}\omega \]
or projectively as a {\em polarity} on $\mathbb{P}V$, identifying (see Definition~\ref{def:Polarity})
\[ \mathbb{P}V \to \mathbb{P}V^*, ~~q\mapsto \xi_q. \]
Through this identification any hypersurface $\Sigma\subset\mathbb{P}V$ has a corresponding {\em polar surface}, $\Sigma'\subset \mathbb{P}V$,  upon identifying its dual $\{ \text{Ann}( T_q\Sigma ): q\in\Sigma\} \subset\mathbb{P}V^*$ with a locus in $\mathbb{P}V$ through the above polarity. In summary:
\begin{defi}\label{def:PolSurf}
    The {\em polar surface}, $\Sigma'\subset\mathbb{P}V$, with respect to $\omega$ of a hypersurface $\Sigma\subset \mathbb{P}V$ is the locus of contact points of $\Sigma$:
    \[ \Sigma' := \{ (T_q\Sigma)^\omega: q\in \Sigma\} \subset \mathbb{P}V. \]
\end{defi}

The polar surface of a table $\Sigma\subset\mathbb{P}V$ plays a key role in its induced 
outer contact billiards.


\section{Examples}\label{sec:Exs}

Before presenting some general results, \S \ref{sec:GenThms}, we examine several examples 
in the first nontrivial case, when $\dim V = 4$. So, we will denote here $(\R^4,\omega)$ 
as some 4-dimensional linear symplectic vector space with contact reduction
$(\mathbb{RP}^3,\xi)$ and will focus our attention on quadratic tables
\[ \Sigma=\{Q=0\}\subset\mathbb{RP}^3,\]
where $Q:\R^4\to\R$ is some indefinite quadratic form. For such tables, their outer contact
billiard admits an explicit description: the domain of the correspondence is filled by 
invariant quadrics, with an induced dynamics on each leaf determined by iterating a fixed 
linear transformation (see Propositions \ref{prop:QuadRP3}, \ref{prop:QuadR4} below). We 
first present those examples which we noticed to have some distinguished properties. The 
remaining types of quadratic surfaces and their contact dynamics are given explicitly in 
Section~\ref{sec:QuadExs} from Appendix \ref{sec:Quads} below. \\



Unless mentioned otherwise, we work in appropriate symplectic coordinates
\begin{equation}
\label{eq:StdSymp}
    \R^4\ni(x,y,u,v),\qquad \omega=dx\wedge dy+du\wedge dv,
\end{equation}
and, whenever convenient, identify 
\begin{equation}\label{eq:CplxCoord}
    \R^4\ni (x,y,u,v)\longleftrightarrow(x+iy,u+iv)=(z,w)\in \C^2.
\end{equation}

\subsection{Clifford tori}\label{sec:ClTori}

We consider first a table of the form
\[ \Sigma_c = \{ |z|^2 = c|w|^2 \} \]
in the complex coordinates \eqref{eq:CplxCoord}, for $c>0$ some
constant. This table projects not only under real scaling to a torus in $\mathbb{RP}^3$, but also under complex scaling (the Hopf map) to a circle in $\mathbb{CP}^1$.

First, we determine the polar surface of $\Sigma_c$. From Proposition \ref{prop:SGrad}, we compute the symplectic gradient: the contact point of $q = \text{span}\{ (z,w)\}\in\Sigma_c$ is $q' = \text{span}\{ (iz,-icw)\}\in\Sigma'$ so that
\[ \Sigma' = \{ c|z|^2 = |w|^2\}.\]




To visualize the correspondence, we pass to the double cover 
$S^3 \stackrel{2:1}{\longrightarrow } \mathbb{RP}^3 = S^3/\{\pm 1\}$. The 
intersection $\Sigma_c\cap S^3$ is the Clifford torus
\[\boldsymbol{\Sigma} = \{(c_1e^{i\theta}, c_2e^{i\varphi}) \in \C^2 : \theta, \varphi \in S^1\} \subset S^3,\]
where the constants $c_1, c_2$ such that $c= (c_1/c_2)^2$ are normalized to satisfy 
$c_1^2+c_2^2=1$. The projections of $\boldsymbol{\Sigma}$ on each complex coordinate are 
circles of radii $c_1$ and $c_2$, respectively. Likewise, 

\[\boldsymbol{\Sigma}'=\{(c_2e^{i\theta}, c_1e^{i\varphi}) \in \C^2 : \theta, \varphi \in S^1\} \subset S^3.\]

Two points $\boldsymbol{Z}, \boldsymbol{Z}' \in S^3$ are in outer contact correspondence (see Definition \ref{def:OutContBill}) if there exist $\boldsymbol{q} \in \boldsymbol{\Sigma}$ with corresponding $\boldsymbol{q}'\in \boldsymbol{\Sigma}'$ and $t\in \R$ such 
that
\begin{equation}\label{eq:CliffCorr}
    \boldsymbol{Z} = \cos(t)\boldsymbol{q} + \sin(t)\boldsymbol{q}', \qquad
\boldsymbol{Z'} = \cos(t)\boldsymbol{q} - \sin(t)\boldsymbol{q}'.
\end{equation}
Projecting onto each complex coordinate, the characteristic geodesic becomes a tangent 
ellipse to the projections of $\boldsymbol{\Sigma}$ and $\boldsymbol{\Sigma'}$, as shown 
in Figure \ref{fig:Clifford}.

\begin{figure}
    \centering
    \includegraphics[width=0.9\linewidth]{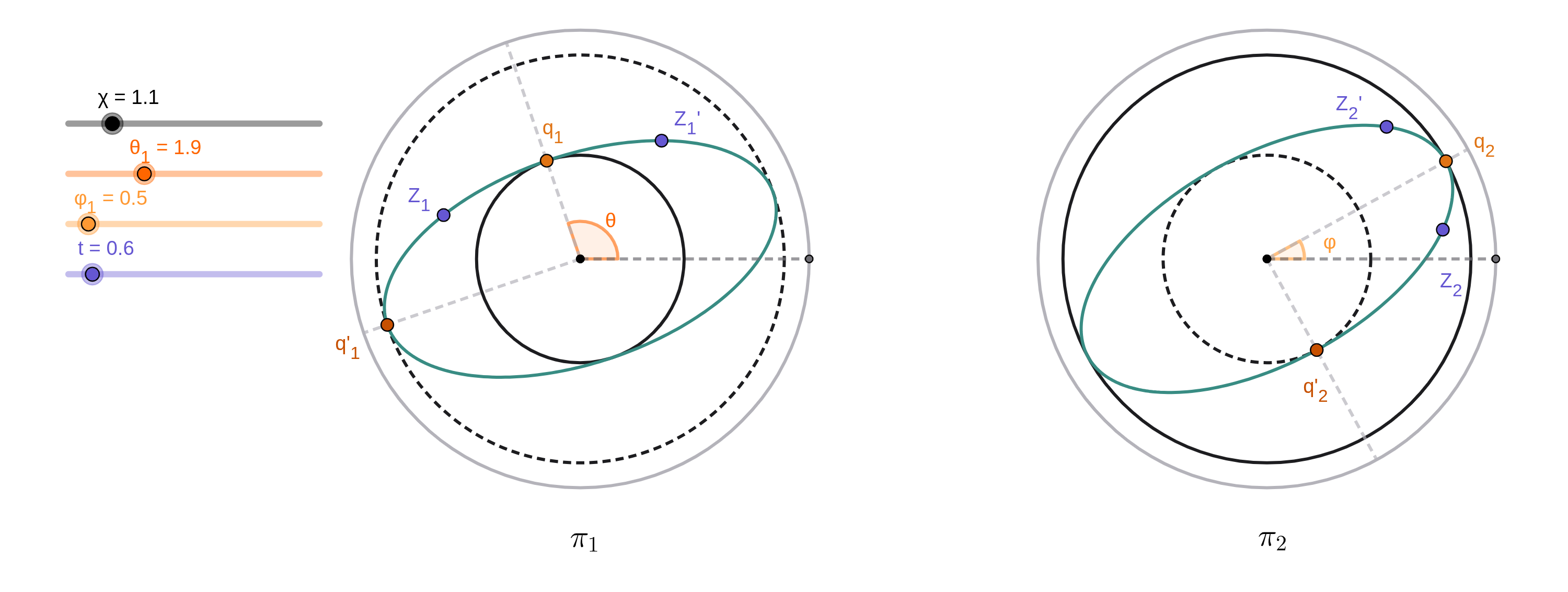}
    \caption{Outer contact billiard on a Clifford torus. The left and right images show 
    the projections onto the first and second complex coordinates, respectively.    
    The points $\boldsymbol{Z}=(Z_1, Z_2)$ and $\boldsymbol{Z'}=(Z_1', Z_2')$ are in 
    outer contact billiard correspondence (note: the norms $|Z_1|,|Z_2|$ are not independent but rather related by $|Z_2|^2 = 1-|Z_1|^2$). An interactive version of this figure is 
    available at \href{https://www.geogebra.org/calculator/g5pzdhyh}{GeoGebra}.}
    \label{fig:Clifford}
\end{figure}


\begin{rmk}
When $c_1=c_2$, every characteristic line is contained in $\Sigma_1$, i.e., $\Sigma_1$ is `self-dual'. Consequently, the 
domain of the correspondence reduces to $\Sigma_1$, a torus ruled by characteristic lines, and any pair of points lying on a common ruling are in outer 
contact billiard correspondence.    
\end{rmk}

For $c_1 \neq c_2$, the correspondence has the following properties:
\begin{enumerate}
    \item The correspondence is defined on the region bounded by $\Sigma_c$ and $\Sigma'$.
    \item Two points $\boldsymbol{Z}, \boldsymbol{Z'} \in S^3$ are in outer contact 
    correspondence if and only if they lie on the same characteristic geodesic and are 
    equidistant from the tangency point. This mirrors the definition of outer symplectic 
    billiards described in \cite{tabachnikov1995dual}. As we shall prove in 
    Section~\ref{sec:Hopf}, this equivalence characterizes Hopf tables.
    \item Every point of the domain is in correspondence with exactly two points. The next 
    point in the orbit is given by the rotations on each complex coordinate through angles 
    $\alpha,\beta$, where
    \begin{equation}\label{eq:CTorAngles}
        \cos\alpha = \frac{c^2\cos^2 t - \sin^2 t}{c^2 \cos^2t +  \sin^2 t}, \qquad \cos\beta = \frac{\cos^2 t - c^2\sin^2 t}{ \cos^2t + c^2 \sin^2 t}
    \end{equation} 
    which is just the condition that $\boldsymbol{q}\in \boldsymbol{\Sigma}$ in \eqref{eq:CliffCorr}.

    \item Thus, every orbit is contained in a Clifford torus, and so the domain of the 
    correspondence is foliated by invariant Clifford tori.
    \end{enumerate}

The explicit description of the dynamics allows us to determine when a Clifford torus admits periodic orbits. This already illustrates a difference compared to the symplectic counterpart, where non-degenerate odd-periodic orbits are guaranteed to exist if the table is a closed submanifold (Theorem 2 in \cite{albers2024outer}).

\begin{prop}
For almost all values of $c$, the outer contact billiard correspondence with respect to $\Sigma_c$ admits no periodic orbits.
\end{prop}

\begin{proof}
Indeed, eliminating $t$ from eqs.~\eqref{eq:CTorAngles}, periodic orbits on a Clifford torus  
correspond to values $c = (c_1/c_2)^2$ and $\alpha = q_1/p_1, \; \beta =q_2/p_2\in \Q$ satisfying
    \[ c^2 = \left( \frac{1 + \cos\alpha}{1-\cos\alpha}\right)\left( \frac{1 - \cos\beta}{1+\cos\beta}\right), 
    ~~\text{or}~~,~ 
    c = \left|\dfrac{\tan\frac{\beta}{2}}{\tan \frac{\alpha}{2}}\right|.\]
As we let $q_j/p_j\in \Q$ range over the rationals, we have a countable 
number of parameter values, $c$, for Clifford tori $\Sigma_c$ that admit periodic contact billiard orbits, while the remaining 
admit only quasi-periodic orbits.
\end{proof}

\begin{rmk}
    The relation 
    \begin{equation}
    \label{eq:KeplerProblem}
        \dfrac{\tan^2 \frac{\beta}{2}}{\tan^2 \frac{\alpha}{2}} = c^2
    \end{equation} 
    between the admissible angles $\alpha,\beta$ on a given Clifford torus with parameter 
    $c = cst.$, is the same as the relation between the true and eccentric anomalies used 
    to parametrize elliptic orbits of the Kepler problem. Geometrically, take a circle and 
    ellipse (of eccentricity $e = \left|\frac{c^2-1}{c^2+1}\right|$) sharing a common 
    diameter. Then, for $c^2>1$, the admissible angles are related via the figure:
    \begin{center}
        \begin{tikzpicture}[scale=.7]
    
    \draw (-3,0) -- (3,0); 

    \draw[thick] (3,0) arc (0:180:3 and 2); 
    \draw[thick] (3,0) arc (0:180:3 and 3); 

    \filldraw[black] (0,0) circle (1pt); 
    \filldraw[black] (2.24,0) circle (1pt); 

    \filldraw[black] (1.5,1.73) circle (2pt); 
    \filldraw[black] (1.5,2.6) circle (2pt); 

    \draw (1.5,0) -- (1.5,2.6);
    \draw (2.24,0) -- (1.5,1.73);
    \draw (0,0) -- (1.5,2.6);


    \node at (2.5,.3) {$\beta$};
    \node at (.4,.2) {$\alpha$};
    
    \node at (2.24,-.3) {focus};
    \node at (0,-.3) {center};
    
\end{tikzpicture}
    \end{center}
whereas for $c^2\in (0,1)$ the roles of $\alpha,\beta$ are interchanged. 
\end{rmk}

It is natural to ask for the smallest possible period. As we will see in 
Section~\ref{sec:periodic}, this billiard correspondence does not admit 3-periodic orbits 
when $\dim V=4$ (Proposition \ref{prop:No3RP3}). However, the Clifford torus determined by 
the quadratic form
\[ \frac{x^2+y^2}{\sqrt{2}-1} - \frac{u^2+v^2}{\sqrt{2}+1} = 0,\]
admits 4-periodic orbits. Figure \ref{fig:4periodic} shows the periodic orbit with 
vertices:
\[ (1:1:1:1), ~ (0:-1:0:1), ~(-1:1:-1:1), ~ (1:0:-1:0).  \]
Not only that, this is the only table defined by a quadratic form admitting $4$-periodic 
orbits (see Proposition \ref{prop:4PerQuad} below).

\begin{center}
    \includegraphics[width=0.9\linewidth]{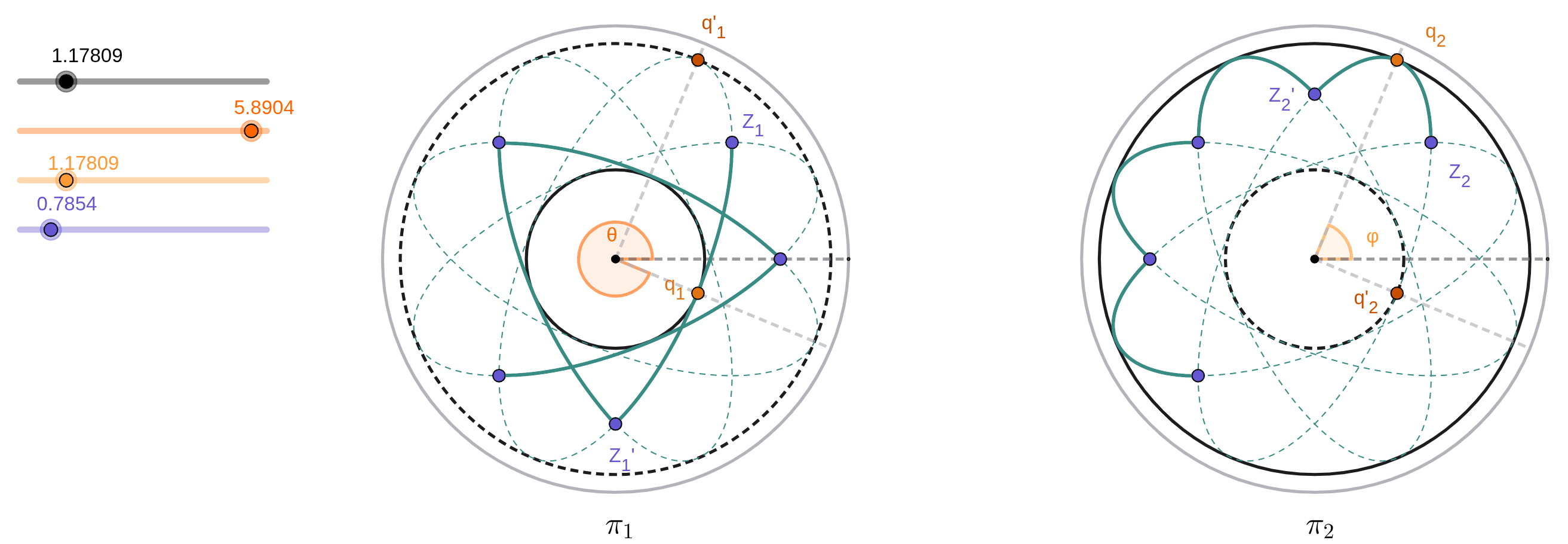}
    \figcaption{A 4-periodic orbit on the Clifford torus. Here, 
    $\alpha=-3\pi/4, \beta = \pi/4$. After $4$ iterations, the initial point $(Z_1,Z_2)$ 
    is its antipodal $(-Z_1, -Z_2)$.}
    \label{fig:4periodic}
\end{center}

\begin{rmk}
The projections of the characteristic lines onto each complex coordinate coincide with the 
trajectories of mechanical billiards (see e.g.~\cite{zhao2024mechanical}) along Hooke orbits (centered conics) bouncing off of certain centered circles. 
\end{rmk}

Summarizing, we have the following explicit description of the correspondence, as a special case of Proposition \ref{prop:QuadR4}.

\begin{prop}
\label{prop:ClTori}
The domain of the outer contact billiard map over $\Sigma_c = \{ x^2 + y^2 = c(u^2 + v^2)
\}\subset\mathbb{RP}^3$, with $c>0, c\ne 1$, is the region foliated by the Clifford tori:
\[ I_{\sigma^2} = \left\{ \sigma^2 = \frac{x^2+ y^2 - c(u^2 - v^2)}{c( u^2 + v^2 - c(x^2 + y^2))}   \right\} \subset\mathbb{RP}^3 \]
for $\sigma \in (0,\infty)$. Each of these tori is invariant under the outer contact 
billiard map over $\Sigma_c$. The map is given, on $I_{\sigma^2}$ by:
\[ I_{\sigma^2}\ni (x:y:u:v)\mapsto (x':y':u':v') \in I_{\sigma^2} \]
where
\[ \begin{pmatrix}
        x' \\ y' \\ u' \\ v'
    \end{pmatrix} = 
    \begin{pmatrix}
        \cos \alpha & -\sin \alpha & 0 & 0 \\
        \sin \alpha & \cos \alpha & 0 & 0 \\
        0 & 0 & \cos \beta & -\sin \beta \\
        0 & 0 & \sin \beta & \cos \beta
    \end{pmatrix}
    \begin{pmatrix}
        x \\ y \\ u \\ v
    \end{pmatrix}.\]
The angles $\alpha, \beta$ are constant on $I_{\sigma^2}$, determined by 
$\sigma = \sqrt{\sigma^2}> 0$ through:
\[ \frac{(1-\sigma^2, 2\sigma)}{1+\sigma^2} = (\cos\alpha, \sin\alpha), ~~\frac{(1-c^2\sigma^2, -2c\sigma)}{1 + c^2\sigma^2} = (\cos\beta, \sin\beta)\]
and satisfying $c^2 = \frac{\tan^2\frac{\beta}{2}}{\tan^2\frac{\alpha}{2}}$.
\end{prop}


\subsection{Ellipsoids}\label{sec:Sphere}

Consider the table 
\[ E_a = \{ a(x^2 + y^2) + u^2 = v^2 \}, ~~a > 0,\]
which we see as an ellipsoid in the affine chart $\{ v = 1 \}$. In particular, 
this example does have some singular points which we are curious to examine. We find that 
the domain of the outer contact billiards map over $E_a$ is filled by invariant 
ellipsoids, on each of which the dynamics is determined by iterating a certain fixed linear 
map (acting as a product of a usual rotation and a hyperbolic rotation). Namely (from 
Proposition \ref{prop:QuadR4}) we have explicitly: 

\begin{prop}
\label{prop:ellipsoid}
The domain of the outer contact billiard map over the ellipsoid $E_a= \{ a(x^2 + y^2) + u^2 = v^2 \}\subset \mathbb{RP}^3$, with $a>0$, is the region filled by the quadratics:
\[ I_{\sigma^2} = \left\{ \sigma^2 = \frac{a( x^2 + y^2) + u^2 - v^2}{(x^2 + y^2)/a + v^2 - u^2} \right\} \subset\mathbb{RP}^3\]
for $\sigma\in (0,\infty)$. The singular tangent planes of $E_a$ constitute $I_{a^2}$. 
Each of these quadratic surfaces is invariant under the outer contact billiard map over 
$E_a$. When $\sigma^2\ne a^2$, the map is given on $I_{\sigma^2}$ by:
    \[ I_{\sigma^2}\ni (x:y:u:v)\mapsto (x':y':u':v') \in I_{\sigma^2} \]
    for
    \[ \begin{pmatrix}
    x' \\ y' \\ u' \\ v'
    \end{pmatrix} = \begin{pmatrix}
        \cos s & -\sin s & 0 & 0 \\
        \sin s & \cos s& 0 & 0 \\
        0 & 0 & \epsilon\cosh t & \epsilon\sinh t \\
        0 & 0 & \epsilon\sinh t & \epsilon\cosh t
    \end{pmatrix}\begin{pmatrix}
    x \\ y \\ u \\ v
    \end{pmatrix},\]
where $\epsilon = \text{sgn}\{ a^2 - \sigma^2 \}$, and the angles $s,t$ are constant on 
$I_{\sigma^2}$, determined through:
    \[ \frac{(1-\sigma^2, 2\sigma)}{1+\sigma^2} = (\cos s, \sin s), ~~\frac{(a^2 + \sigma^2, 2a\sigma)}{|a^2 - \sigma^2|} = (\cosh t, \sinh t). \]
This table admits no periodic orbits.
\end{prop}

This ellipsoid example exhibits exactly two singular points where $q_\pm=q_{\pm}'$, namely:
\[ q_\pm = (0: 0 : \pm 1 :  1), \qquad \xi_{q_\pm} = T_{q_\pm}E_a = \{u=\pm v\}.\]
The union of these singular tangent planes constituting the singular leaf 
$I_{a^2} = \xi_{q_+}\cup \xi_{q_-}$. Since, for points not on these singular planes, we 
have the billiard map explicitly, we can easily examine the limiting behavior induced on 
these singular planes. For this example, as we let a point limit to a position on one of 
these singular planes, we find well-defined limiting positions for its correspondents. 
Namely, tending to some point
\[ Z_+ = (x_+:y_+: 1:1)\in \xi_{q_+}\backslash \{ q_+\},\]
we find its `regularized' correspondents are:
\[ (0:0:1:1) = q_+\in \xi_{q_+}, \text{ and }\left( 2 \frac{x_+\cos \alpha_* + y_+ \sin\alpha_*}{(x_+^2 + y_+^2)\sin \alpha_*} : 2 \frac{-x_+\sin \alpha_* + y_+ \cos\alpha_*}{(x_+^2 + y_+^2)\sin \alpha_*} : -1:1\right)\in \xi_{q_-}, \]
where $\tan \alpha_* = \frac{2a}{1-a^2}$.
Likewise, a point $Z_- = (x_-:y_-:-1:1)\in \xi_{q_-}\backslash \{ q_-\}$ has correspondents 
\[(0:0:-1:1) = q_- \in \xi_{q_-}, \text{ and } \left( 2 \frac{x_-\cos \alpha_* - y_- \sin\alpha_*}{(x_-^2 + y_-^2)\sin \alpha_*} : 2 \frac{x_-\sin \alpha_* + y_- \cos\alpha_*}{(x_-^2 + y_-^2)\sin \alpha_*} : 1:1\right) \in \xi_{q_+}\]
whereas a point on the intersection of the two singular planes has limiting correspondents the two singular points: $q_\pm$.
We see then that one singular point acts as an attractor and the other as a repeller (which 
is also clear from the explicit description in Proposition \ref{prop:ellipsoid}). 
Otherwise, this limiting correspondence between the two singular planes is a rotation by 
the fixed angle $\alpha_*$ together with an inversion over a certain circle of radius 
$2/\sin\alpha_*$. We think that examining the behavior around singular points in general 
will be an interesting theme for future work (see Question 8 in Section~\ref{sec:questions} 
of open questions below).

\begin{figure}[h]
    \centering
    \includegraphics[width=0.45\linewidth]{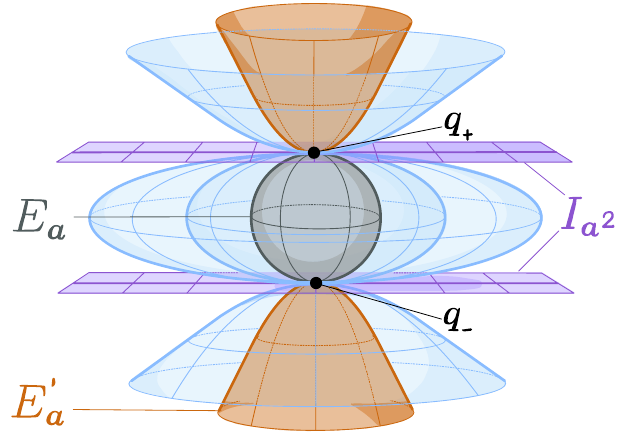}
    \caption{Singular foliation of invariant leaves on an affine chart. We show the pencil 
    of quadratics $I_{\sigma^2}$ and distinguish three special ones: the table, $E_a$, its 
    polar surface $E_a'$ and the union of the singular tangents $I_{a^2}$.}
    \label{fig:singular-foliation}
\end{figure}

\begin{rmk}
The pencil of quadratics $I_{\sigma^2}$ form a singular foliation. Namely, each member of 
this family of surfaces is tangent at the singular points $q_\pm$ (see Figure 
\ref{fig:singular-foliation}).  
\end{rmk}

\subsection{Sheared torus}\label{sec:TwistTori}

Consider the table given by the quadric \[ \Sigma = \{  v(v+x) = y(u-y) \},\]
which is topologically a torus. Here the table and its polar intersect along a distinguished common line, $\ell_*$. The outer billiard dynamics takes place along a family of invariant quadratic leaves, all intersecting along this same distinguished line. In this case, apart from collinear orbits contained in $\ell_*$, there are no `genuine' periodic orbits. Explicitly, the dynamics is given on each leaf by iterating a certain fixed linear map: 

\begin{prop}
\label{prop:TwistTori}
    The domain of the outer contact billiard map over the torus $\Sigma = \{  v(v+x) = y(u-y) \} \subset \mathbb{RP}^3$ is the region foliated by the quadrics
    \[I_{\sigma^2}=\left\{ \sigma^2 = \frac{y^2 + v^2 + xv-uy}{y^2 + v^2 - xv + uy} \right\}\subset \mathbb{RP}^3\]
    for $\sigma\in (0,\infty)$. Each of these quadratic surfaces is invariant under the outer contact billiard map over $\Sigma$. On each leaf, the map is given by 
    \[ I_{\sigma^2}\ni (x: y: u: v)\mapsto (x':y':u':v')\in I_{\sigma^2}\]
    as
     \[\begin{pmatrix}
        x'\\
        y'\\
        u'\\
        v'
    \end{pmatrix} = \begin{pmatrix}
        \cos s & \sin 2s & -\sin s & \cos 2s - 1\\
        0 & \cos s & 0 & -\sin s\\
        \sin s & 1 - \cos 2s & \cos s & \sin 2s\\
        0 & \sin s & 0 & \cos s
    \end{pmatrix} \begin{pmatrix}
        x \\
        y \\
        u \\
        v
    \end{pmatrix},\]
    where the angle $s$ is  determined via
    \[\dfrac{(1-\sigma^2, 2\sigma)}{1+\sigma^2} = (\cos s,\sin s)\] and is constant on $I_{\sigma^2}$.  
\end{prop}

\begin{prop}
   Apart from collinear orbits on the distinguished line $\ell_*=\{ y =v=0 \}$, the table $\Sigma = \{ v(v+x)=y(u-y)\}$ admits no periodic orbits.
\end{prop}
\begin{proof}
    In the complex coordinates $\zeta =x + iu, \eta= y+iv$, the map is given by the (real) projectivization of:
    \[ (\zeta,\eta)\mapsto  (\zeta',\eta') = e^{is}(\zeta +2\sin s~\eta, \eta).\]
    Where for points not on $\Sigma$, we have $\eta\ne 0$ and $\sin s\ne 0$. If there were a periodic orbit, it would project to a periodic orbit under $\mathbb{C}^2\to \mathbb{CP}^1$. But in the affine chart
    \[(\zeta:\eta)\longleftrightarrow \zeta/\eta = z \]
    this map is simply $z\mapsto z +2\sin s$, and has no periodic orbits.
\end{proof}

On this distinguished line, there are periodic orbits of any period: any two points are in contact correspondence. Generally, when we speak of genuine periodic orbits, we will have non-collinear ones in mind (see Remark \ref{rmk:collinear}).


\subsection{Cylinders (pinched tori)}\label{sec:Cyl}

Now we will consider a degenerate quadratic surface:
\[ \Sigma = \{ x^2 + y^2 = v^2 \}\subset\mathbb{RP}^3, \]
topologically a pinched torus with a singular point at $(0:0:1:0)$, which we see as a standard cylinder in the affine chart $\{ v = 1\}$. For this table, the dynamics splits into a product of usual outer billiards in a plane together with a shift along the cylinder's axis. More explicitly:
\begin{prop}
    For the table $\Sigma = \{ x^2 + y^2 = v^2\}$, the outer contact billiard map is given in the affine chart $(x,y,u) \longleftrightarrow (x:y:u:1)$ as follows: $(x,y,u)$ and $(x',y',u')$ are in correspondence over $\Sigma$ when (see Figure \ref{fig:Supp})
    \begin{enumerate}
        \item $(x,y)$ and $(x',y')$ are in symplectic outer billiard correspondence over the unit circle and,
        \item the vertical displacement $u'-u$ is the product $pd$ of the distance $d$ between $(x,y), (x',y')$, with the (signed) distance $p$ from the origin to the line joining $(x,y)$ and $(x',y')$.
    \end{enumerate}
\end{prop}
\begin{proof}
    The proof is essentially the same for a generalized cylinder, given by $f(x,y) = cst.$, in the affine chart $\{ v = 1\}$, where the contact structure is $\xi|_{(x,y,u)} = \{ du = xdy - ydx\}$. Note that the contact points of the vertical tangent planes lie at infinity, so that the harmonic conjugate condition is given in this affine chart by taking usual affine midpoints along the characteristic lines (see \S \ref{sec:affine} for the more general situation). One computes that the characteristic lines along the generalized cylinder $\{ f(x,y) = cst.\}$ are directed by: \[ {\bf v} = \frac{-f_y\partial_x + f_x\partial_y + (xf_x +yf_y)\partial_u}{\sqrt{f_x^2 + f_y^2}}\] 
    where $\frac{-f_y\partial_x + f_x\partial_y}{\sqrt{f_x^2 + f_y^2}}$ directs the unit tangent to the level set $f(x,y) = cst.$, while $p = \frac{xf_x +yf_y}{\sqrt{f_x^2 + f_y^2}}$ is the signed distance of this tangent to the origin so that $(x'-x,y'-y,u'-u) = {\bf v} d$ and the vertical shift $u'-u$ is as claimed.
\end{proof}

    \begin{center}
        \includegraphics[width=0.5\linewidth]{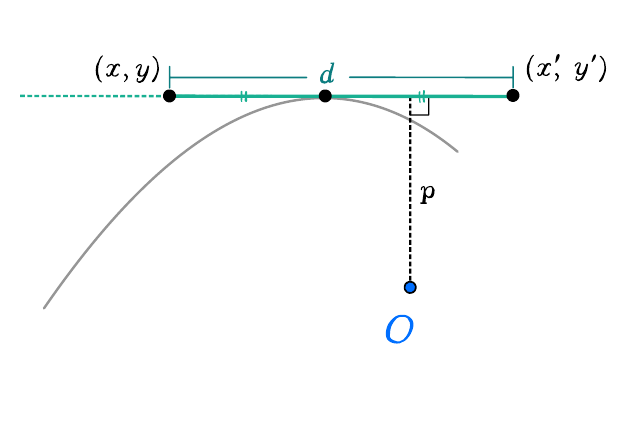}
        \figcaption{The points $(x, y)$, $(x',y')$ are in outer correspondence over $\{f(x,y)=cst\}$. The support function of the tangent line connecting $(x,y), (x',y')$ is $p$.}
        \label{fig:Supp}
    \end{center}

\begin{rmk}
    Note that, in this case, the harmonic conjugate condition over this certain table has reduced in this affine chart to the usual notion of midpoint. 
    We will see in Section \ref{sec:affine} 
    (Theorem \ref{thm:AffTable}) the reciprocal result: if the notion of harmonic conjugate 
    and midpoint coincide in some affine chart, then the table must be a ``generalized 
    cone''.
\end{rmk}

\begin{rmk}
The vertical displacement along an orbit is determined by the net accumulation of the successive signed areas, $pd$, over the orbit.  For a general closed cone, one can always choose an appropriate symplectic basis so that in the affine chart $\{ v = 1\}$, the $u$-axis is contained in the interior of the corresponding generalized cylinder. In particular, the cumulative vertical displacements are then strictly increasing and there can be no periodic orbits. The singular point of the generalized cone acts as a repeller/attractor (as if one had glued together the singular poles on the ellipsoid example \ref{sec:Sphere} above).

\end{rmk}


\section{Periodic orbits}\label{sec:periodic}

The periodic trajectories in our quadratic examples above are surprisingly rare. This 
raises the question of which periods can occur for the outer contact billiard 
correspondence. In this section, we consider some basic properties of periodic outer 
contact billiards orbits when $\dim V  = 4, \mathbb{P}V = \mathbb{RP}^3$.\\

Let $(Z_j)_{j\in \Z}$ be an orbit of the contact billiard correspondence. Every pair of 
consecutive points $Z_j$, $Z_{j+1}$ is joined, by definition, by a characteristic tangent 
line, and is therefore contained in a Legendrian line of $(\mathbb{P}V,\xi)$ (see Figure 
\ref{fig:contactcorrespondence}). Consequently, every periodic orbit determines vertices 
of some closed Legendrian polygon (a discrete Legendrian knot). A related notion 
(equivalent when $\dim V=4$) is explored in \cite{ConOvsLagrConf}.


\begin{center}
    \centering
     \includegraphics[width=0.4\linewidth]{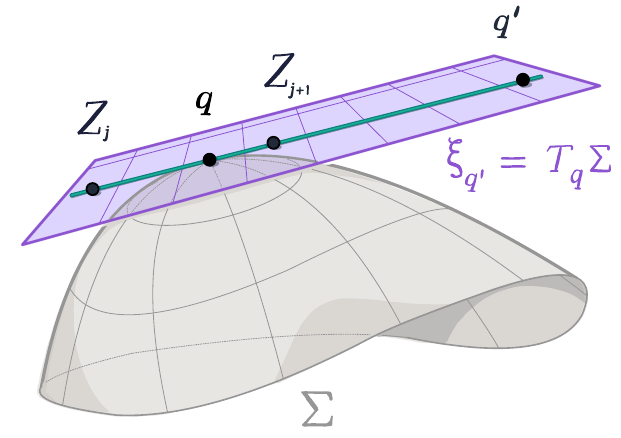}
    \figcaption{If $q\in Z_j\wedge Z_{j+1}$ is the reflecting point on the table with contact point $q'$ then $\xi_{q'}\supset Z_j\wedge Z_{j+1} = q\wedge q' \ni q'$.}
    \label{fig:contactcorrespondence}
\end{center}



This geometric interpretation already imposes strong restrictions on the possible periods. In contrast to outer symplectic billiards, where 3-periodic orbits occur over any compact table (Theorem 2 of \cite{albers2024outer}), for outer contact billiards:

\begin{prop}\label{prop:No3RP3}
    For any table $\Sigma \subset (\mathbb{RP}^3,\xi)$, the associated outer contact 
    billiard correspondence admits no non-collinear 3-periodic orbits. 
\end{prop}
\begin{proof}
    A non-collinear 3-periodic outer contact billiard orbit would have $Z_1, Z_2, Z_3$ at 
    the vertices of some Legendrian triangle in $(\mathbb{RP}^3, \xi)$. There are no 
    Legendrian triangles here as the following argument shows: a Legendrian triangle would 
    lift to an independent triple $\boldsymbol{Z_1}, \boldsymbol{Z_2}, \boldsymbol{Z_3} \in
    (\R^4,\omega)$, with $\omega(\boldsymbol{Z_1}, \boldsymbol{Z_2}) = 
    \omega(\boldsymbol{Z_2}, \boldsymbol{Z_3}) = \omega(\boldsymbol{Z_3}, \boldsymbol{Z_1}) 
    = 0$, and we would have a 3-dimensional isotropic subspace spanned by the 
    $\boldsymbol{Z}_j$'s. But $\omega$ is non-degenerate, and its largest isotropic 
    subspaces in $\R^4$ are the Lagrangian 2-planes.
\end{proof}

\begin{rmk}
\label{rmk:collinear}
    The non-collinearity assumption is necessary. Indeed, fix some Legendrian line 
    $L\subset \xi_{q'}$ through a given point $q'$ and choose three distinct points 
    $Z_1,Z_2, Z_3\in L\backslash q'$. For each pair $(Z_j, Z_{j+1})$, take the reflection 
    point $q_j\in L$ as the harmonic conjugate of $q'$ over $Z_j, Z_{j+1}$. Any smooth 
    hypersurface tangent to $\xi_{q'}$ at the points $q_j\in L\subset \xi_{q'}$ admits, by 
    construction, $Z_1, Z_2, Z_3$ as a collinear 3-periodic orbit. By the same 
    construction, for any $n\in \mathbb{N}$, there exist smooth tables admitting collinear $n$-periodic orbits as well. 
\end{rmk}

On the other hand, non-degenerate 4-periodic outer contact billiards orbits do exist (as 
we have seen in Section \ref{sec:ClTori}). Let us briefly comment on the non-collinear 
Legendrian quadrilaterals in $(\mathbb{RP}^3,\xi)$, all of which are equivalent under 
projectivized symplectic transformations, $\text{PSp}_4(\R)$, i.e., in an appropriately 
adapted symplectic basis. Indeed, by Proposition~\ref{prop:No3RP3}, the vertices cannot 
lie in a projective plane, and therefore admit linearly independent lifts 
$\boldsymbol{Z_1}, \ldots , \boldsymbol{Z_4} \in \R^4$. As consecutive vertices are joined 
by Legendrian lines we have $\omega(\boldsymbol{Z_j}, \boldsymbol{Z_{j+1}})=0$ (where the 
subindices are read modulo four). After rescaling the lifts, we may further assume that 
$\omega(\boldsymbol{Z_1}, \boldsymbol{Z_3}) = \omega(\boldsymbol{Z_2}, \boldsymbol{Z_4}) 
=1,$ all other pairings being zero. In this symplectic basis, ${\boldsymbol{Z}}_1, \ldots, 
{\boldsymbol{Z}}_4$, we have $\omega = dx\wedge dy + du\wedge dv$ for 
\[ Z_1 = \text{span}\{ \partial_x \}, ~Z_2 = \text{span}\{ \partial_u\}, ~Z_3 = \text{span}\{ \partial_y\}, ~Z_4 = \text{span}\{ \partial_v\}. \]

To construct a table admitting this $4$-periodic orbit, it suffices to choose reflection 
points $q_j\in Z_j\wedge Z_{j+1}\backslash \{ Z_j, Z_{j+1}\}$ on each edge (see Figure 
\ref{fig:contactcorrespondence}). Such a choice determines the corresponding contact point 
$q_j'$ as the harmonic conjugate of $q_j$ with respect to $Z_j, Z_{j+1}$. The table 
$\Sigma$ must contain the points $q_j$, and satisfy $T_{q_j}\Sigma = \xi_{q_j'}$. Among 
arbitrary smooth tables, satisfying these conditions presents no obstacle. On the other 
hand, the situation is more rigid within a prescribed family of tables where it is not 
immediate whether one can always fit such a table to the given 4-periodic orbit. For the 
family of quadratic tables, we find:


\begin{prop}\label{prop:4PerQuad}
    The {\em only} quadratic table in $\mathbb{RP}^3$ admitting a non-degenerate 4-periodic 
    orbit is a Clifford torus given by:
    \[   x^2 + y^2 = \frac{\sqrt{2}-1}{\sqrt{2} + 1}(u^2 + v^2) \]
    in appropriate symplectic coordinates ($\omega = dx\wedge dy + du\wedge dv$). It admits 
    a 2-parameter family (invariant torus) of 4-periodic orbits (see Figure 
    \ref{fig:4periodic} above). For example:
\[ (1:1:1:1), ~ (0:-1:0:1), ~(-1:1:-1:1), ~ (1:0:-1:0).  \]
\end{prop}

\begin{proof}
Suppose that a quadratic table $\Sigma = \{ Q = 0\}$, where $Q$ is a quadratic form on 
$\R^4$, admits a non-collinear 4-periodic orbit. By the discussion preceding the 
proposition we may consider symplectic coordinates with  
\[Z_1 = (1:0:0:0), Z_2 = (0:0:1:0), Z_3 = (0:1:0:0), Z_4 = (0:0:0:1),\]
and $\omega = dx\wedge dy + du\wedge dv$. 
Our general quadratic form is then given in these coordinates by
\[Q = Ax^2 + By^2 + Cu^2 + Dv^2 + 2\alpha yu + 2\beta xu + 2\gamma xy + 2axv + 2byv + 2cuv.\]
For the quadratic surface $\Sigma = \{Q = 0\}$ to admit this orbit, we require some 
reflecting points
    $$q_j = \text{span}\{ \Lambda_j \boldsymbol{Z_j} + (1-\Lambda_j)\boldsymbol{Z_{j+1}}\}\in \Sigma, \quad \Lambda_j\ne 0,1$$
together with their harmonic conjugates
$$q_j' = \text{span}\{ \Lambda_j \boldsymbol{Z}_j + (\Lambda_j - 1)\boldsymbol{Z}_{j+1}\}$$
over $Z_j, Z_{j+1}$ to satisfy $T_{q_j}\Sigma = \xi_{q_j'}$. These conditions are a system 
of equations where the unknowns are the coefficients of $Q$ together with the four 
parameters $\Lambda_j$. At first sight, this appears to be an overdetermined problem: 
prescribing four points and four tangent hyperplanes on a quadratic typically imposes 
$4\cdot 3+ 4=16$ conditions, whereas there are only $9+4=13$ unknowns (up to an overall 
scaling of $Q$). As it turns out, many of these conditions are redundant here because of 
the special configuration of our points and hyperplanes under consideration. 

Recall that $Q$ is homogeneous of degree two. For $X_Q$ the symplectic gradient of $Q$ 
then
\[    \omega(\boldsymbol{x},X_Q(\boldsymbol{x})) = d_{\boldsymbol{x}}Q({\boldsymbol{x}}) = 2Q({\boldsymbol{x}}) \]
So that the tangency condition $T_q\Sigma = \xi_{q'}$, ie $\omega({\boldsymbol{q}}, 
{\boldsymbol{q}}') = 0$ for ${\boldsymbol{q}}' = X_Q({\boldsymbol{q}})$, already implies 
$q\in \Sigma$.


Therefore, it only remains to impose the tangency conditions $T_{q_j}\Sigma = \ker d_{q_j}Q = \xi _{q_j'}$. For $q = (x:y:u:v)\in \Sigma$, we compute its contact point, $\text{span}\{ X_Q({\boldsymbol{q}})\}$, is at:
     \[ q' = (-(By+\alpha u + \gamma x + bv): Ax+\beta u + \gamma y + av : -(Dv+ax+by+cu): Cu + \alpha y + \beta x  + cv).\]
     So, upon setting
     \[ \Lambda_j' := 1-\Lambda_j\ne 0,1 \]
     the tangency conditions $T_{q_j}\Sigma = \xi_{q_j'}$, consist in the following system:
     \[   (\Lambda_1:0: -\Lambda_1':0) = (-(\alpha \Lambda_1' + \gamma \Lambda_1): A\Lambda_1+\beta \Lambda_1' : -(a\Lambda_1 +c\Lambda_1'): C\Lambda_1' + \beta \Lambda_1), \]
     \[  (0:-\Lambda_2':\Lambda_2:0) = (-(B\Lambda_2'+\alpha \Lambda_2): \beta \Lambda_2 + \gamma \Lambda_2' : -(b\Lambda_2'+c\Lambda_2): C\Lambda_2+ \alpha \Lambda_2' ), \]
     \[ (0:\Lambda_3:0:-\Lambda_3') = (-(B\Lambda_3  + b\Lambda_3'):  \gamma \Lambda_3 + a\Lambda_3' : -(D\Lambda_3'+b\Lambda_3):  \alpha \Lambda_3 + c\Lambda_3'),\]
     \[  (-\Lambda_4':0:0:\Lambda_4) = (-( \gamma \Lambda_4' + b\Lambda_4): A\Lambda_4' + a\Lambda_4 : -(D\Lambda_4 +a\Lambda_4' ):  \beta \Lambda_4'  + c\Lambda_4).\]
     The system admits exactly two families of solutions:
      \begin{enumerate}
         \item $\Lambda_1\Lambda_2\Lambda_3\Lambda_4 = \Lambda_1'\Lambda_2'\Lambda_3'\Lambda_4'$, with
    \[ Q_+ = \left(x + \frac{\Lambda_1\Lambda_2}{\Lambda_1'\Lambda_2'}y - \frac{\Lambda_1}{\Lambda_1'}u - \frac{\Lambda_4'}{\Lambda_4}v \right)^2  \]
         \item $\Lambda_1\Lambda_3 = \Lambda_1'\Lambda_3'$, and, $\Lambda_2'\Lambda_4' =- \Lambda_2\Lambda_4$, with 
         \[ Q_- = x^2 +\left( \frac{\Lambda_1\Lambda_2}{\Lambda_1'\Lambda_2'}y\right)^2 + \left( \frac{\Lambda_1}{\Lambda_1'}u\right)^2 + \left( \frac{\Lambda_4'}{\Lambda_4}v\right)^2 - 2\frac{\Lambda_1}{\Lambda_1'}\left( \frac{\Lambda_1\Lambda_2}{\Lambda_1'\Lambda_2'}yu + ux  \right) -  2\frac{\Lambda_4'}{\Lambda_4}\left( xv + \frac{\Lambda_3'\Lambda_4'}{\Lambda_3\Lambda_4}yv  \right).  \]     
     \end{enumerate}
The first type of quadratic is degenerate: $Q_+ = 0$ just defines a hyperplane with no 
outer contact billiards dynamics. The second type of quadratic, $Q_- = 0$, upon 
diagonalizing is exactly the special Clifford torus stated above: set $\lambda_j = 
\Lambda_j/\Lambda_j'$ and consider the change of basis
\[ X + U = \frac{x-\lambda_1\lambda_2 y}{\sqrt{2}}, ~~V + Y = \frac{x + \lambda_1\lambda_2 y}{\sqrt{2}} \]
\[ V-Y = \lambda_1 u, ~~X-U = \lambda_2 v.\]

\end{proof}
\begin{rmk}
    From this last proof there is a 2-parameter family of quadratic surfaces admitting a 
    given fixed Legendrian quadrilateral as a 4-periodic orbit. The family is parametrized 
    by selecting any pair of points on adjacent sides of the Legendrian quadrilateral 
    through which said surface must pass (the remaining pair of points through which the 
    quadratic passes on the opposite sides are then determined uniquely through the 
    conditions $\Lambda_1\Lambda_3 = \Lambda_1'\Lambda_3'$, and, $\Lambda_2'\Lambda_4' =- 
    \Lambda_2\Lambda_4$ in item 2 above).
\end{rmk}


\section{Preserved contact structure}\label{sec:GenThms}

In this section, we establish some general properties of the outer contact billiard map. 
Our first result shows that the name ``contact billiards'' is justified.

\begin{thm}\label{thm:ContTransf}
    When a local diffeomorphism, the outer contact billiard map is a contactomorphism. 
\end{thm}

That is, given a regular hypersurface $\Sigma\subset (\mathbb{P}V,\xi)$, the outer contact billiards correspondence over $\Sigma$ (Definition \ref{def:OutContBill}) defines a locus in $\mathbb{P}V \times \mathbb{P}V$ consisting of those pairs in outer contact billiards correspondence over $\Sigma$. When smooth and transverse to the fibers, this locus is given locally by the graph of a contactomorphism (see Remark \ref{rmk:ContGraph} below).

\begin{proof}
    Let $\Sigma\subset (\mathbb{P}V, \xi)$ be a regular hypersurface where $\xi$ is a standard 
    contact structure on $\mathbb{P}V$ induced by a linear symplectic structure 
    $(V,\omega)$. Suppose that the outer contact billiard correspondence determined by 
    $\Sigma$ is a local diffeomorphism on an appropriate domain, and consider a smooth 
    curve of tuples
    \[ t\mapsto (Z(t), q(t), Z'(t), q'(t)) \]
    in outer contact billiards correspondence over $\Sigma$. Since $Z,Z'$ are harmonic 
    conjugates over $q, q'$, we may choose lifts $\boldsymbol{Z}, \boldsymbol{q}, 
    \boldsymbol{Z'}, \boldsymbol{q'}\in V$ such that
    \[ \boldsymbol{Z} = \boldsymbol{q} + \boldsymbol{q'}, \quad \boldsymbol{Z'} = \boldsymbol{q} - \boldsymbol{q}'. \]
    Write $\boldsymbol{\dot{Z}}, \boldsymbol{\dot{q}}, \boldsymbol{\dot{Z'}}$, and 
    $\boldsymbol{\dot{q'}}$ for the corresponding derivatives.
    
    Since $q\in T_q\Sigma = \xi_{q'}$, we have
    \[ \omega(\boldsymbol{q}, \boldsymbol{q'}) = 0 \implies \omega(\boldsymbol{\dot{q}}, \boldsymbol{q'}) = \omega(\boldsymbol{\dot{q}'}, \boldsymbol{q}). \]
    In fact, since $\dot q\in T_q\Sigma = \xi_{q'}$ we have:
    \begin{equation}
        \label{eq:omega0}
        \omega(\boldsymbol{\dot{q}}, \boldsymbol{q'}) = 0 = \omega(\boldsymbol{\dot{q'}}, \boldsymbol{q}).
    \end{equation}
    
   Using Equation \ref{eq:omega0}, we compute
    \[ \omega(\boldsymbol{Z},\boldsymbol{\dot{Z}}) = \omega(\boldsymbol{q} + \boldsymbol{q'}, \boldsymbol{\dot{q}} + \boldsymbol{\dot{q'}}) = \omega(\boldsymbol{q}, \boldsymbol{\dot{q}}) + \omega(\boldsymbol{q'},\boldsymbol{\dot{q'}}) = \omega(\boldsymbol{q} - \boldsymbol{q'}, \boldsymbol{\dot{q}} - \boldsymbol{\dot{q'}}) = \omega(\boldsymbol{Z'}, \boldsymbol{\dot{Z'}}), \]
    and, in particular: 
    \[ \dot Z\in \xi_Z \iff \omega(\boldsymbol{Z},\boldsymbol{\dot{Z}}) = 0  = \omega(\boldsymbol{Z'}, \boldsymbol{\dot{Z'}}) \iff \dot Z'\in \xi_{Z'}.\] 
\end{proof}

Examining the proof more closely, one finds a stronger statement: in a certain sense, the 
outer contact billiards map (Definition \ref{def:OutContBill}) is the {\em only} such 
reflection rule generating contact transformations over any table.
\begin{corr}\label{corr:Uniqueness}
    Consider a non-trivial outer billiard reflection rule on $(\mathbb{P}V,\xi)$ which 
    reflects $Z$ to $Z'$ over characteristic lines of tables $\Sigma\subset\mathbb{P}V$. 
    Then the induced maps are local contact transformations if and only if they are given by 
    the reflection rule described in Definition \ref{def:OutContBill} above.
\end{corr}
\begin{proof}
    Using the notation of the previous proof, suppose the reflection rule is given by 
    $\boldsymbol{Z'} = \boldsymbol{q} + \beta \boldsymbol{q'}$, where the coefficient 
    $\beta$ is given by some function determining the ``twisted'' reflection rule. Then we 
    have: $\omega(\boldsymbol{Z'}, \boldsymbol{\dot{Z'}}) = \omega(\boldsymbol{Z}, 
    \boldsymbol{\dot{Z}}) + (\beta^2 - 1)\omega(\boldsymbol{q'}, \boldsymbol{\dot{q'}})$. 
    We generate contact transformations over arbitrary tables when 
    $0 = \omega(\boldsymbol{Z}, \boldsymbol{\dot{Z}}) \iff 0 = \omega(\boldsymbol{Z'}, \boldsymbol{\dot{Z'}})$. That is, when $0 = (\beta^2 - 1)\omega(\boldsymbol{q'}, \boldsymbol{\dot{q'}})$. Since we are free to move on a general table, 
    $\omega(\boldsymbol{q'}, \boldsymbol{\dot{q'}})$ can be arbitrary, and the only such 
    choice of reflection rule is with $\beta = \pm 1$, i.e.~either that of
    Definition \ref{def:OutContBill} (when $\beta = -1$) or the trivial reflection rule, 
    $Z = Z'$ (when $\beta = 1$).
\end{proof}

\begin{rmk}\label{rmk:ContGraph}
    Given two contact manifolds $(N,\xi), (N',\xi')$ their product inherits a natural pair of hyperplane distributions $\pi^*\xi = \{ v\in T(N\times N') : \pi_*v \in \xi\}$ and $(\pi')^*(\xi')$ for $\pi, \pi'$ the standard projections onto each factor. A submanifold $\Gamma\subset N\times N'$ such that $(T\Gamma)\cap (\pi^*\xi) = (T\Gamma)\cap ((\pi')^*\xi')$ is, when transverse to the fibers, given locally by the graph of a contactomorphism.  
\end{rmk}


\subsection{Affine tables}
\label{sec:affine}

We note that harmonic conjugate reduces to taking midpoints in an affine chart when one of 
the points in the tuple is a point at infinity. Our outer contact billiards takes place in 
projective space, but naturally we would like to view it also in some affine chart. In 
such an affine chart, there appears a less complicated reflection rule since we now have a 
notion of midpoints.
\begin{center}
    \includegraphics[width=0.4\linewidth]{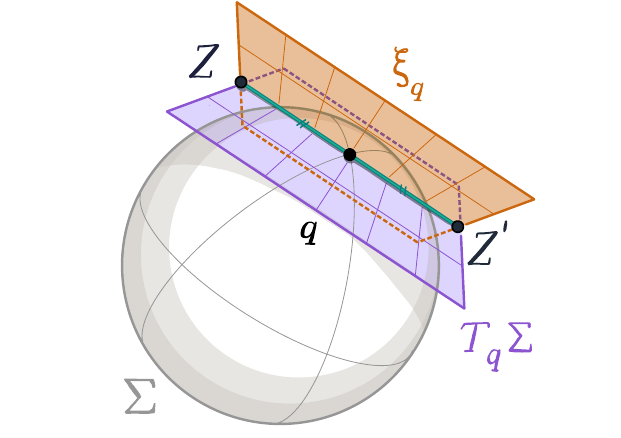}
    \figcaption{The midpoint reflection rule in an affine chart: $q\in \Sigma$ is the midpoint of $Z,Z'$}
\end{center}
Typically our outer contact billiards reflection and this midpoint rule will {\em not} 
coincide, however we may still ask: for which tables do these two rules coincide? When 
could this more familiar midpoint reflection rule generate contact transformations?

As it turns out, such tables must be of the type considered in our examples from 
\S \ref{sec:Cyl} above:

\begin{thm}\label{thm:AffTable}
    Let $\Pi_\infty\in \mathbb{P}V^*$ be a fixed hyperplane as the plane at infinity of an 
    affine chart. Then the following are equivalent:
    \begin{enumerate}
        \item the midpoint reflection rule over a table $\Sigma\subset\mathbb{P}V$ in this 
        affine chart generates a contact transformation of $(\mathbb{P}V,\xi)$,
        \item the midpoint reflection rule over $\Sigma$ coincides with the outer contact 
        billiards reflection rule over $\Sigma$ (Definition ~\ref{def:OutContBill}),
        \item $\Sigma$ is a generalized cone with vertex at the contact point $q_\infty\in \Pi_\infty$ of the plane at infinity (the point where $\Pi_\infty = \xi_{q_\infty}$).
    \end{enumerate} 
\end{thm}
\begin{proof}
The equivalence of $(1) \iff (2)$ follows from Corollary \ref{corr:Uniqueness}. 

To see the equivalence $(2)\iff (3)$, we note that the midpoint reflection rule is
precisely harmonic conjugation with respect to a point at infinity. Therefore, the 
midpoint rule coincides with that of Definition~\ref{def:OutContBill} if and only if, for 
every $q\in \Sigma$, the contact point $q'=(T_q\Sigma)^\omega$ lies on the hyperplane at 
infinity $\Pi_\infty$, or equivalently, $(T_q\Sigma)^\omega\subset \Pi_\infty$. By taking 
symplectic complements we have $T_q\Sigma \supset (\Pi_\infty)^\omega = q_\infty$ so that 
such planes are exactly those belonging to the pencil of hyperplanes incident to 
$q_\infty$.
Hence, every tangent hyperplane to $\Sigma$ passes through the fixed point $q_\infty$. 
These are precisely the generalized cones in item (3), see Figure \ref{fig:affineCone}.

\begin{figure}
    \centering
    \includegraphics[width=0.5\linewidth]{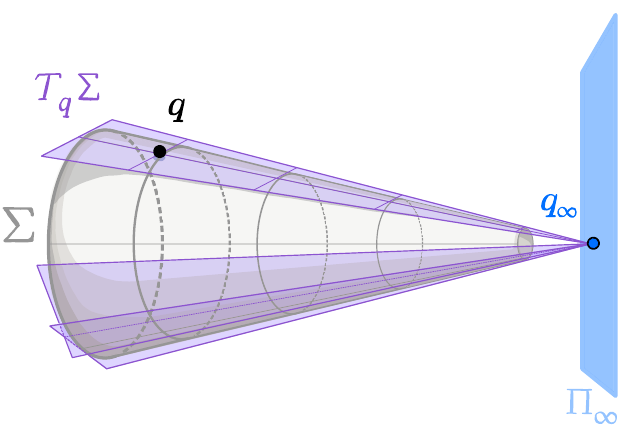}
    \caption{Generalized cone with vertex at infinity.}
    \label{fig:affineCone}
\end{figure}

\end{proof}

\subsection{Hopf tables}
\label{sec:Hopf}

It is also standard when studying the projective space to pass to a double cover, by 
considering oriented lines (rays), and making some choice of inner product in order to 
work then on an appropriate unit sphere. Similarly to our affine chart of the last 
section, after such choices there now appears a simpler reflection rule, by using the 
inner product to reflect with equal angles along characteristic lines.
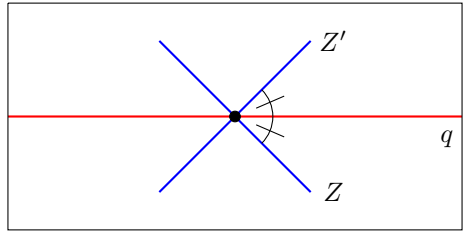
\begin{figure}[h]
\centering
\begin{tikzpicture}  
    
    \draw[red,thick] (-3,0) -- (3,0); 
    
    \draw[blue,thick] (-1,1) -- (1, -1); 
    \draw[blue,thick] (-1,-1) -- (1, 1); 

    \filldraw[black] (0,0) circle (2pt);

    \draw (-3,1.5) -- (-3,-1.5) -- (3,-1.5) -- (3,1.5) -- (-3,1.5);

    \node at (2.8,-.3) {$q$};
    \node at (1.3,-1) {$Z$};
    \node at (1.3,1) {$Z'$};


    \draw (.5,0) arc (0:45:.5);
    \draw (.5,0) arc (0:-45:.5);

    \draw (0.28,0.11) -- (0.65,0.27);
    \draw (0.28,-0.11) -- (0.65,-0.27);
\end{tikzpicture}

    \caption{The angle reflection rule with respect to an inner product: $\angle(q,Z) = \angle(q,Z')$.}
    \label{fig:AngRef}
\end{figure}

Typically our outer contact billiards reflection and this angle reflection rule will 
{\em not} coincide, however we may ask again: for which tables do these two rules 
coincide? When could this more familiar equal angles reflection rule generate contact 
transformations?

As it turns out such tables must be of the type considered in our example from \S \ref{sec:ClTori} above:

\begin{thm}\label{Thm:CplxTable}
    Let $J:V\to V$ be some $\omega$-compatible complex structure, $J^*\omega = \omega$,
    with associated inner product $g_J(u,v)= \omega(u,Jv)$. Then the following are 
    equivalent:
    \begin{enumerate}
        \item the angle reflection rule over a table $\Sigma\subset\mathbb{P}V$ with 
        respect to the inner product $g_J$ generates a contact transformation of 
        $(\mathbb{P}V,\xi)$,
        \item the angle reflection rule over $\Sigma$ with respect to $g_J$ coincides with 
        the outer contact billiards reflection rule over $\Sigma$ (Definition 
        \ref{def:OutContBill}),
        \item $\Sigma$ is a {\em Hopf table}, that is to say, $\Sigma = \pi_J^{-1}(\Gamma)$ 
        for some hypersurface $\Gamma\subset \mathbb{CP}^{n-1}$ and $\pi_J:V\to \mathbb{P}_\C V \cong \mathbb{CP}^{n-1}$ the Hopf map associated to the complex
        structure $J$ on $V\cong \C^n$.
    \end{enumerate}
\end{thm}
\begin{proof}
    The equivalence of $(1) \iff (2)$ is Corollary \ref{corr:Uniqueness}. \\
    To see the equivalence between (2) and (3), let $q\in \mathbb{P}V$ be a line through 
    the origin, and $\Pi\in \mathbb{P}V^*$ a hyperplane containing $q$ so that
    \[ \Pi \supset q\implies \xi_q = q^\omega \supset \Pi^\omega \]
    and, when $\Pi\ne q^\omega$, the characteristic line of $\Pi\ni q$ is the 2-plane 
    spanned by $q$ and $\Pi^\omega$. The equal angles reflection rule with respect to 
    $g_J$ will coincide with the outer contact billiard reflection rule over $\Pi$ if and 
    only if $q$ and $\Pi^\omega$ are $g_J$ orthogonal subspaces (see Figure 
    \ref{fig:ContactSpherical}).
    
    \begin{figure}[h]
    \centering
    \includegraphics[width=0.4\linewidth]{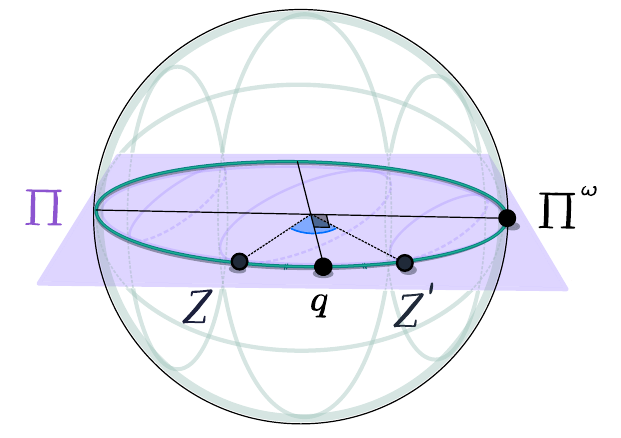}
    \caption{Harmonic conjugates and equal angles coincide when the reflecting point 
    $q\in \Pi$ and $\Pi^\omega$ are perpendicular.}
    \label{fig:ContactSpherical}
\end{figure}

    That is, for $\Pi^\omega = \text{span}\{ u\}$, we have:
    \[ 0 = g(u,q) = \omega(u,Jq) \iff Jq\in (\Pi^\omega)^\omega = \Pi. \]
    Applied to the tangent planes $\Pi= T_q\Sigma$ of some table $\Sigma$, we see that the 
    equal angles reflection rule over the table coincides with the outer contact billiard 
    reflection rule if and only if $Jq\in T_q\Sigma$. Since $Jq$ is the infinitesimal 
    generator of the $S^1$-action $q\mapsto e^{it}q$, the condition $Jq\in T_q\Sigma$ 
    means that $\Sigma$ is invariant under this action. Hence $\Sigma$ is a union of 
    fibers of the associated Hopf fibration.
\end{proof}

\begin{rmk}
    In this last result, it is not necessary that the associated inner product $g_J$ be definite, so that midpoint reflection rules in hyperbolic spaces can also be included.
\end{rmk}

\section{Open questions}
\label{sec:questions}

The theory developed here leaves many further questions to be explored. Some concern the 
geometric foundations of outer contact billiards, while others point towards possible 
extensions and connections with classical billiard theory. We collect a few of these 
problems below.

\begin{enumerate}
    \item \textbf{Local contactomorphisms as contact billiards.} Theorem 
    \ref{thm:ContTransf} shows that when the outer contact billiard correspondence is a 
    local diffeomorphism, the map preserves the contact structure. Can every local 
    contactomorphism be realized as an outer contact billiard map?
    
    \item \textbf{Inner contact billiards.} Throughout this paper, we focus on \emph{outer} 
    billiards. Is there a natural \emph{inner} contact billiard? What kind of structure 
    should these preserve? 

    For instance, let us consider $(\mathbb{RP}^3, \xi)$ where it is possible to cook up 
    various candidate reflection laws for incident lines to a given surface. The reflection 
    law off of said surface then induces a certain local transformation of $\text{Gr}(2,\mathbb{R}^4)$. What is an appropriate analogue of Theorem \ref{thm:ContTransf} to 
    decide which reflection law shall earn the name of inner contact billiards?

    For example, a candidate reflection law in $(\mathbb{RP}^3, \xi)$ can be defined 
    (using only the projective and contact structure) in the following way:
    \begin{itemize}
        \item Given a line $\ell$ incident to a hyperplane $\Pi$ at $q\in \Pi \backslash \{ q'\}$ (for $q' = \Pi^\omega$ the contact point of $\Pi$).
        \item Let $n$ be the contact point of the plane through $q'$ and $\ell$, and 
        $\nu = q\wedge n$ the line through $q$ and $n$.
        \item Then consider $\ell'$ to be the reflection of $\ell$ off of $\Pi$ when: 
            \begin{itemize}
                \item[(a)] $\ell'$ lies in the pencil of lines generated by $\ell, \nu$, and 
                \item[(b)] $\ell, \ell'$ are harmonic conjugates with respect to $q\wedge q', \nu$. 
            \end{itemize}
    \end{itemize}
    See Figure \ref{fig:inner-contact}.

    \begin{figure}
        \centering
        \includegraphics[width=0.45\linewidth]{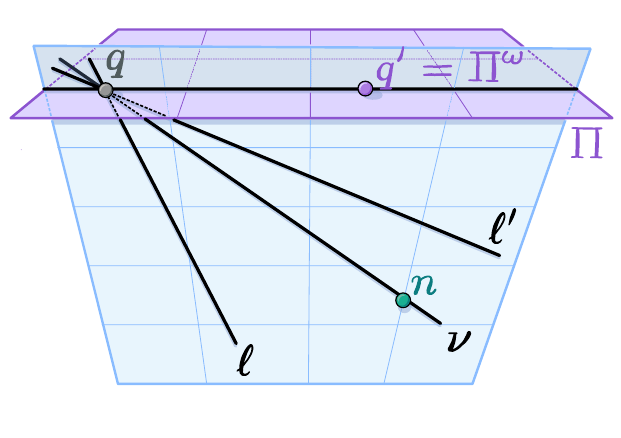}
        \caption{The proposed reflection law defined using only projective and contact structure. Here the line $\nu$ plays the same role as the transversal line field in projective billiards (see \cite{tabachnikovProjective}). However, $\nu$ depends not only on the reflecting plane $\Pi$ and incidence point $q$, but also on the incident line $\ell$ (this normal remaining fixed among incident lines in the pencil through $\ell,q\wedge q'$). }
        \label{fig:inner-contact}
    \end{figure}
    
    \item \textbf{Variational principle.} The outer contact billiard construction is 
    defined simultaneously from a contact and a projective structure. Is there a natural 
    variational principle associated with this geometry? More specifically, can outer 
    contact billiard trajectories be characterized as critical points of a suitable action 
    functional in a similar way that Herglotz' variational principle characterizes 
    trajectories of contact vector fields? See for example \cite{DiscHerg}, where a 
    discrete Herglotz variational principle (generating contact transformations) has been 
    developed.
    
    \item \textbf{Projective duality.} Is there a natural dual construction for outer 
    contact billiards under projective duality? 
    
    \item \textbf{Higher-dimensional quadrics.} Our explicit description of outer contact 
    billiards over quadratic tables was restricted to $\mathbb{RP}^3$. Is the outer 
    contact billiard map over quadratic surfaces still integrable in higher dimensions? 
    See Remark \ref{rmk:quadrics-higher-dim} in the Appendix below.

    \item \textbf{Higher degree tables.} The contact billiards over quadratic tables in 
    $(\mathbb{RP}^3, \xi)$ has been described here completely. What can one say about 
    higher degree tables? 
    
    Some comments:
\begin{itemize}
    \item The characteristic lines along a general surface in $\mathbb{RP}^3$ form a 2-parameter family of lines, or {\em congruence} of lines, in $\mathbb{RP}^3$ (see \cite{DarbouxII}: Book IV).
    
    \item For quadratic surfaces in $\mathbb{RP}^3$, their outer contact billiards induce 
    a bona fide dynamics: through each point there pass at most {\em two} characteristic 
    lines (one a `future' and the other a `past'). For higher degree tables, this need not 
    be so. Indeed, for a degree $d$ surface in $\mathbb{RP}^3$ typically there will pass 
    $d(d-1)$ characteristic lines through a given point (the congruence of characteristic 
    lines has {\em order} $d(d-1)$ in the terminology of \cite{DarbouxII}), and one will need to consider a type of `branched dynamics' as we have seen in remark \ref{rmk:quadrics-higher-dim}.

    \item A given congruence of lines generally envelope a certain focal surface, which 
    may be composed of various branches. A given line may touch the focal surface at 
    multiple points (each point along some branch). Apart from the table and its polar, 
    what sorts of focal surfaces are associated with a congruence of characteristic lines 
    along an algebraic table in $(\mathbb{RP}^3,\xi)$? What sorts of surface 
    transformations and sequences of conjugate nets (see \S 5 of \cite{ChernSurfTry}) can 
    arise from such congruences composed of characteristic lines?
    
\end{itemize}

    \item \textbf{Higher-codimension tables.} The outer contact billiard construction 
    admits a natural extension beyond tables that are hypersurfaces (codimension one). For 
    a submanifold of codimension greater than one, each tangent space determines a family 
    of characteristic lines, rather than a single one, and we have still an induced outer 
    contact billiards correspondence. What new phenomena emerge in this higher-
    codimensional setting?
    
    \item \textbf{Singular points.} Singular points arise in many classes of projective 
    tables, yet their role in the dynamics requires further investigation (see \S 
    \ref{sec:Sphere}). 
    What sort of behaviors can appear around these singular points? When is it possible to 
    regularize the correspondence through such singularities? 
    
    \item \textbf{Relations to dissipative billiards.} Contact geometry provides a broad 
    framework for dissipative dynamics, and so the outer contact billiards here can be 
    considered as a certain type of dissipative billiard (we see this dissipation in 
    several of our examples already, e.g.~the attracting and repelling points in Example 
    \ref{sec:Sphere}). 
    
    Dissipative billiards are usually considered as certain conformally symplectic 
    transformations in their corresponding phase space (see \cite{pinball}, 
    \cite{dissipative-billiards}, and \cite{diss-symp}). Conformal symplectic 
    transformations are special cases of contact transformations, lifting, locally at 
    least, to contact transformations of an appropriate contactification. Thus one can 
    imagine that these usual dissipative billiards should correspond to an appropriate 
    projection of contact billiards. Can contact billiards help understand these usual 
    dissipative billiards?

    Separately, one can mimic the definition of dissipative symplectic billiards in the context of contact billiards. Instead of asking for the cross ratio in Definition \ref{def:OutContBill} to be $-1$, one can replace it with a different fixed value $\beta \neq -1$. By Corollary \ref{corr:Uniqueness}, we know such billiards do not preserve the contact structure. Is there any preserved structure?
    
    
    \item \textbf{Classical billiard conjectures.} In the case of Birkhoff (Euclidean) 
    billiards, there are two ``holy grails'': the Ivrii and Birkhoff-Poritsky conjectures. 
    One can also pose the analogs of these conjectures in this new setting:
    \begin{itemize}
        \item (\textit{Ivrii} \cite{ivrii}) Does the set of periodic contact billiard 
        orbits have measure zero?
        \item (\textit{Birkhoff-Poritsky} \cite{birkhoff-poritsky}) Given a hypersurface 
        $\Sigma$ such that the outer contact billiard correspondence is integrable, can we 
        ensure $\Sigma$ is a quadratic table? 
    \end{itemize}
      
    \item \textbf{Periodic Clifford tori.} We showed that the dynamics on Clifford tori, Section \ref{sec:ClTori}, are completely determined by two rotation frequencies satisfying an explicit relation (Equation \eqref{eq:KeplerProblem}). 
    For a given Clifford torus, which periods can actually occur? 
\end{enumerate}

\begin{appendix}
    \section{Quadratic tables}\label{sec:Quads}

Our examples above, \S \ref{sec:Exs}, concern quadratic tables (in $(\mathbb{RP}^3,\xi)$). The space of quadratic forms on a symplectic vector space is well studied (see \cite{WillNormForms}, or \S 2.4 of \cite{EDSIV}). For completeness, we recall here that modulo symplectic changes of basis there is a 2-parameter family of quadratic forms on $(\R^4,\omega)$, in correspondence with the adjoint orbits of $\mathfrak{sp}_4(\R)$. Consequently, modulo re-scalings, we have a 1-parameter family of quadratic tables in $(\mathbb{RP}^3,\xi)$. Explicitly:
\begin{prop}[Special case of \cite{WillNormForms}]\label{prop:QuadRP3}
 A non-trivial (i.e.~non-empty, and non-planar) quadratic surface in $(\mathbb{RP}^3,\xi)$ is given, in an appropriate symplectic basis $(x,y,u,v)\in \R^4$ with $\omega = dx\wedge dy + du\wedge dv$ by one of the following normal forms:
    \begin{enumerate}
        \item $\{ x^2 + y^2  = v^2\}, \{ xy = v^2\}, ~\{ vy = x^2\}$,
        \item $\{ cxy = uv\}$
        \item $\{ c(x^2 + y^2) = u^2 + v^2\}$,
        \item $\{ c(x^2 + y^2) + u^2 = v^2\}$,
        \item $\{ x(cy + v) = u(y-cv)\}$,
        \item $\{ xy + uv = uy \}$,
        \item $\{  v(v+x) = y(u-y) \}$
    \end{enumerate}
    for $c\in \R\backslash 0$ a parameter.
\end{prop}
\begin{proof} We summarize a derivation. Consider our non-degenerate symplectic form as an invertible skew-linear map $\omega:\R^4\to \R^{4*}$. Given a quadratic form, $Q:\R^4\to \R$, let $S:\R^4\to \R^{4*}$ be the (symmetric) map for $Q$'s underlying inner product: $Q({\bf v}) = (S{\bf v},{\bf v})$. Then:
\[ X = \omega^{-1} S \in \mathfrak{sp}_4(\R)  \]
(the linear vector field ${\bf v}\mapsto X{\bf v}$ is, up to constant multiples, the symplectic gradient of the quadratic Hamiltonian ${\bf v}\mapsto (S{\bf v},{\bf v})$). Classifying quadratic forms up to symplectic changes of basis: $S\mapsto A^*SA$ corresponds to classifying $X\in \mathfrak{sp}_4(\R)$ up to adjoint action, $X\mapsto A^{-1}X A$, for $A\in \text{Sp}_4(\R)$: $A^*\omega A = \omega$. These adjoint orbits are distinguished by the various types of eigenspaces of $X\in \mathfrak{sp}_4(\R)$ (recall: if $\lambda$ is an eigenvalue of $X\in \mathfrak{sp}_4(\R)$, so too is $-\lambda, \bar\lambda, -\bar\lambda$). The items above correspond to the following cases 1: degenerate quadratic forms (some zero eigenvalues); 2: simple real eigenvalues ($a,-a,b,-b$); 3: simple pure imaginary eigenvalues ($\pm ia, \pm ib$); 4: simple real and pure imaginary eigenvalues ($a,-a, \pm ib$); 5: general complex eigenvalue ($a\pm ib, -a\pm ib$); 6: repeated real eigenvalue with a Jordan block; 7: repeated pure imaginary eigenvalue with a Jordan block. Finally, we use the re-scaling freedom to normalize some (non-zero) eigenvalues to one.
\end{proof}

\begin{rmk}
    For a vector space $V$ we denote the natural pairing between $\nu\in V^*, v\in V$ by $(\nu,v)$.
\end{rmk}

\begin{rmk}
    In symplectic coordinates upstairs, $\omega = dx\wedge dy + du\wedge dv$, the induced contact structure is given by (see e.g.~\cite{ContRed}) the kernel of the restriction of 
    \[ \lambda = \iota_{x\partial_x + y\partial_y + u\partial_u + v\partial_v }\omega = -ydx + xdy - vdu + udv \]
    to any surface transversal to the radial direction, e.g.~restriction of $\lambda$ to an affine hyperplane.
\end{rmk}

\begin{rmk}
    The cases with degenerate quadratic forms (item 1) are considered in example \ref{sec:Cyl} above. Item 2 in the list corresponds to the quadratic table described in Proposition \ref{prop:HypTori}.
    Item 3 corresponds to the Clifford tori exposed in example \ref{sec:ClTori}. 
    Item 4 are the ellipsoids, considered in example \ref{sec:Sphere}. Item 5 corresponds to the quadratic case discussed in Proposition \ref{prop:LoxTori}. Item 6 is represented in Proposition \ref{prop:ParTori}.
    Finally, item 7 of the list is precisely the case exposed in example \ref{sec:TwistTori}.
\end{rmk}

Continuing with the notation from the last proof, we compute the following description of outer contact billiards over non-degenerate quadratic surfaces:
\begin{prop}\label{prop:QuadR4}
    Let $\Sigma = \{ {\bf v} : (S{\bf v},{\bf v}) = 0\}\subset\mathbb{RP}^3$ be a non-degenerate quadratic surface with generator
    \[ X = \omega^{-1}S \in \mathfrak{sp}_4(\R).\]
    Then the polar surface of $\Sigma$ with respect to $\omega$ is the quadratic surface $\Sigma' = \{ {\bf v} : (S'{\bf v},{\bf v}) = 0\}\subset\mathbb{RP}^3$ for $X^*S'X = S$, or:
    \[ S' = -\omega S^{-1}\omega.\]
    The domain of the outer contact billiards map over $\Sigma$ (or $\Sigma'$) is the region foliated by the quadratics
    \[ I_{\sigma^2} = \{ {\bf v} : (S{\bf v},{\bf v}) = \sigma^2\det X ~(S'{\bf v},{\bf v}) \}, ~~~\sigma\in (0,\infty) \]
    from the pencil of quadratics through $\Sigma, \Sigma'$. Each quadratic $I_{\sigma^2}$ is invariant under the outer contact billiards map over $\Sigma$ (or $\Sigma'$), explicitly by iterating
    \[ I_{\sigma^2} \ni \text{span}\{ \boldsymbol{Z}\}\mapsto \text{span}\{ (id-\sigma X)(id+\sigma X)^{-1}\boldsymbol{Z} \} \in I_{\sigma^2} \]
    as long as $1/\sigma $ is not an eigenvalue of $X$.
\end{prop}
\begin{proof}
    The polar surface of $\Sigma$ is the locus $\Sigma' = \{ \text{span}\{ X\boldsymbol{q} \}: \text{span}\{ \boldsymbol{q}\} \in\Sigma\}$ so that $S'$ is as claimed. A point $Z = \text{span}\{\boldsymbol{Z}\}$ is in outer contact correspondence with $Z'$ over $\Sigma$ if and only if
    \[ \boldsymbol{Z} = \boldsymbol{q} + \sigma \boldsymbol{q'} = (id + \sigma X)\boldsymbol{q} \]
    for some $q = \text{span}\{ \boldsymbol{q}\}\in \Sigma$ and $\sigma\in\R$, in which case $Z' = \text{span}\{ (id - \sigma X)\boldsymbol{q} \}$. It holds that $q\in \Sigma$ if and only if:
    \[ (S\boldsymbol{Z},\boldsymbol{Z}) = \sigma^2 (S\boldsymbol{q'}, \boldsymbol{q'}) = \sigma^2 \omega(X^3\boldsymbol{q}, \boldsymbol{q}), ~~(S'\boldsymbol{Z}, \boldsymbol{Z}) = (S'\boldsymbol{q}, \boldsymbol{q}) = \omega(\boldsymbol{q},X^{-1}\boldsymbol{q}).  \]
    Likewise, we compute $(S\boldsymbol{Z'}, \boldsymbol{Z'}) =  \sigma^2 (S\boldsymbol{q'}, \boldsymbol{q'}) = (S\boldsymbol{Z}, \boldsymbol{Z})$ and $(S'\boldsymbol{Z'}, \boldsymbol{Z'}) =  (S'\boldsymbol{q}, \boldsymbol{q})  = (S'\boldsymbol{Z}, \boldsymbol{Z})$). Thus the outer contact billiards over a quadratic table admit the integral:
    \[ I(x) = \frac{(S{\bf x}, {\bf x})}{(S'{\bf x},{\bf x})}, ~~~x = \text{span}\{{\bf x}\} \in \mathbb{RP}^3 \]
    or, in other words, the pencil of quadratics through $\Sigma,\Sigma'$ are invariant. To describe the dynamics on each leaf, we restrict to the low-dimensional case. Computing in the basis $X^{-1}\boldsymbol{q}, \boldsymbol{q}, X\boldsymbol{q}, X^2\boldsymbol{q}$ we see
    \[ \omega(X^3\boldsymbol{q}, \boldsymbol{q}) = \det X ~\omega(\boldsymbol{q},X^{-1}\boldsymbol{q})\]
    So that $I = \sigma^2 \det X = cst.$ (for fixed $X \in\mathfrak{sp}_4(\R)$) and the dynamics on this leaf is by iterating the fixed projective transformation stated above ($\sigma = cst.$).
\end{proof}
\begin{rmk}
\label{rmk:quadrics-higher-dim}
    Our explicit description above of the dynamics on the pencil of invariant quadratics is special to the low-dimensional case $(\mathbb{RP}^3, \xi)$. In arbitrary dimensions we still have the invariant pencil of quadratics through $\Sigma,\Sigma'$:
    \[ cst. = I(Z) =  \frac{(S\boldsymbol{Z}, \boldsymbol{Z})}{(S'\boldsymbol{Z},\boldsymbol{Z})} = I(Z'). \]
    However, the relation between the integral $I$ to the dynamics on each leaf (i.e.~the parameter $\sigma$ above) may differ from that in Proposition \ref{prop:QuadR4}, where we had simply $cst.  = \det X  = \frac{I}{\sigma^2}$.

    Namely, outer contact billiards over quadratics in higher dimensions will generally admit more integrals. Moreover, on each leaf (i.e.~fixed level of the integrals), there will not typically be a traditional dynamics but a `branched dynamics', where each point is in contact correspondence with {\em more} than just {\em two} (past/future) points.

    We can already see this difference for the quadratic hypersurface
    \[ \Sigma = \{ xu + ayv + bzw = 0\}\subset\mathbb{RP}^5, ~~\omega = du\wedge dx + dv\wedge dy + dw\wedge dz. \]
    Through a general point in $\mathbb{RP}^5$ there pass up to {\em four} characteristic lines of $\Sigma$, and the outer contact billiards over $\Sigma$ admit {\em two} independent integrals. More explicitly,
   the outer contact billiards over $\Sigma$ admits the general integral $I$ from Proposition \ref{prop:QuadR4} given by:
        \[ I = \dfrac{xu + ayv + bzw}{xu + \frac{yv}{a} + \frac{zw}{b}}. \]
        However, there are more integrals; in fact, we have individually the integrals:
        \[ I_1 = \frac{yv}{xu}, ~~I_2 = \frac{zw}{xu} \]
        so that $I = \frac{1 + a I_1 + b I_2}{1 + I_1/a + I_2/b}$.  On each leaf $I_1 = cst., I_2 = cst.$, the correspondents are given through 
        \[ (x:u:y:v:z:w)\longleftrightarrow  \left(\frac{1-\sigma}{1+\sigma}x: \frac{1+\sigma}{1-\sigma}u:\frac{1-a\sigma}{1+a\sigma}y:\frac{1+a\sigma}{1-a\sigma}v : \frac{1-b\sigma}{1+b\sigma}z:\frac{1+b\sigma}{1-b\sigma}w\right).\]
        where $\sigma$ may take up to four values and, modulo signs, up to {\em two} values. Namely, any root of:
        \[  a^2 b^2(\sigma)^4 - \left(\frac{(a^2 + b^2) + a(1+b^2) I_1 + b(1+a^2)I_2}{1 + \frac{I_1}{a} + \frac{I_2}{b}}\right) (\sigma)^2 + I = 0. \]
        So in higher dimensions, already for quadratic hypersurfaces, the variety of possibilities to describe has become much richer.
\end{rmk}

\subsection{Explicit formulas}\label{sec:QuadExs}

For completeness, we state here explicit formulas for the outer contact billiards induced by the remaining types of quadratic tables in $(\mathbb{RP}^3,\xi)$, only noting that such tables are also integrable in the same sense as the previous examples of quadratic tables.


For item 2 from Proposition \ref{prop:QuadRP3}:

\begin{prop} \label{prop:HypTori}
The domain of the outer contact billiard map over the hyperbolic torus $\Sigma_c = \{ cxy = uv\}\subset\mathbb{RP}^3$, with $c>0, c\ne 1$, is the region foliated by the quadratics:
\[ I_{\sigma^2} = \left\{ \sigma^2=\dfrac{uv - cxy}{c(cuv -xy)} \right\} \subset\mathbb{RP}^3, \]
where $\sigma \in (0,\infty)$. The singular tangent planes of $\Sigma_c$ comprise the degenerate quadratics $I_1, I_{1/c^2}$. Each one of these quadratic surfaces is invariant under the outer contact billiard map over $\Sigma_c$. When $\sigma^2\ne 1, 1/c^2$, the map is given, on $I_{\sigma^2}$ by:
    \[ I_{\sigma^2}\ni (x:y:u:v)\mapsto (x':y':u':v') \in I_{\sigma^2} \]
    with
    \[ \begin{pmatrix}
    x' \\ y' \\ u' \\ v'
    \end{pmatrix} = \begin{pmatrix}
        \lambda_1 & 0 & 0 & 0 \\
        0 & \frac{1}{\lambda_1} & 0 & 0 \\
        0 & 0 & \lambda_2 & 0 \\
        0 & 0 & 0 & \frac{1}{\lambda_2}
    \end{pmatrix}\begin{pmatrix}
    x \\ y \\ u \\ v
    \end{pmatrix},\]
    where $\lambda_1 = \frac{1 + c\sigma}{1- c\sigma}, \lambda_2 = \frac{1-\sigma}{1+\sigma}$, for $\sigma = \sqrt{\sigma^2}>0$, are constant on $I_{\sigma^2}$. Note that since $\sigma\ne 0,  1, 1/c$, we have $\lambda_j\in \R\backslash \{0,\pm 1\}$.
    In particular, this implies that $\Sigma_c$ has no periodic orbits.
\end{prop}

\begin{rmk}
    Here the case when $c = 1$ is exceptional: we have a `self-dual' torus ruled by characteristic lines. There are two ruling lines which are singular (and transversal to the characteristic ruling lines). In this case, the limiting behavior on these singular lines is not completely trivial; it is a 2-periodic map. 
\end{rmk}

\begin{rmk}
    Although Proposition \ref{prop:HypTori} describes the domain as foliated by invariant hyperbolic tori, the closure of these leaves is not disjoint. Any two leaves intersect along four projective lines
    \[(x: 0: u :0),\;  (x: 0 : 0 :v), \;  (0: y : u :0), \; (0: y: 0: v),\]
    all contained in the table $\Sigma_c$. Since these lines are entirely inside the billiard table, removing $\Sigma_c, \Sigma_c'$ leaves a genuine foliation of the billiard domain.
\end{rmk}





Next, item 5 from Proposition \ref{prop:QuadRP3}:

\begin{prop}
\label{prop:LoxTori}
The domain of the outer contact billiard map over the loxodromic torus $\Sigma_c=\{c(xy+uv) = yu -xv \}\subset \mathbb{RP}^3$ is the region foliated by the quadrics:
\[I_{\sigma^2}=\left\{ (1 +c^2)\sigma^2 = \frac{c(xy + uv) + xv - yu}{c(xy + uv) -xv + yu}  \right\}\subset \mathbb{RP}^3\]
for $\sigma\in (0,\infty)$. Each of these quadratic surfaces is invariant under the outer contact billiard map over $\Sigma_c$. On each leaf, the map is given by:
\[I_{\sigma^2} \ni (x: y: u: v) \mapsto (x':y': u': v') \in I_{\sigma^2}\]
by
\[\begin{pmatrix}
    x'\\
    y'\\
    u'\\
    v'
\end{pmatrix} =\begin{pmatrix}
    \rho\cos\varphi & 0 & -\rho \sin\varphi & 0\\
    0  &\frac{\cos\varphi}{\rho} & 0 & -\frac{\sin\varphi}{\rho} \\
    \rho\sin\varphi & 0 & \rho \cos\varphi & 0 \\
    0 & \frac{\sin\varphi}{\rho} & 0 & \frac{\cos\varphi}{\rho}
\end{pmatrix} \begin{pmatrix}
    x \\
    y \\
    u \\
    v
\end{pmatrix},\]
where $\rho, \varphi$ are constant on $I_{\sigma^2}$, determined through:
\[ \rho = \sqrt{1 + \frac{4\sigma c}{\sigma^2 + (1-\sigma c)^2}}, ~~(\cos\varphi, \sin\varphi) = \frac{(1 - (1+c^2)\sigma^2, 2\sigma)}{\sqrt{ (1 - (1 + c^2)\sigma^2)^2 + 4\sigma^2}}, \]
and there are no periodic orbits.

\end{prop}

\begin{rmk}
    The map is more compact in the complex variables $x+iu, y+iv$ where it is given by the (real) projectivization of
    \[ (x+iu, y + iv)\mapsto (\rho e^{i\varphi}(x+iu),  \rho^{-1}e^{i\varphi}(y+iv)).\]
\end{rmk}

\begin{rmk}
    Similarly to the case of hyperbolic tori, although Proposition \ref{prop:LoxTori} describes the domain as foliated by invariant loxodromic tori, these leaves are not disjoint. This time, any two leaves intersect along two disjoint projective lines
    \[(x: 0: u :0)\; \text{ and }  \; (0: y: 0: v),\]
    both contained in the table $\Sigma_c$. Since these lines are entirely inside the billiard table, removing $\Sigma_c, \Sigma_c'$ leaves a genuine foliation of the billiard domain and the table has no singular points.
\end{rmk}



Lastly, item 6 from Proposition \ref{prop:QuadRP3}:

\begin{prop}
\label{prop:ParTori}
    The domain of the outer contact billiard map over the parabolic torus $\Sigma = \{xy+uv = uy\}\subset \mathbb{RP}^3$ is the region foliated by the quadrics  
    \[I_{\sigma^2}=\left\{ \sigma^2 = \frac{xy + uv + uy}{xy + uv - uy} \right\}\subset \mathbb{RP}^3\]
    for $\sigma\in (0,\infty)$. Each of the quadratic surfaces is invariant under the outer contact billiard map over $\Sigma$. On each leaf, the map is given by 
    \[ I_{\sigma^2}\ni (x: y: u: v)\mapsto (x':y':u':v')\in I_{\sigma^2}\]
    as
    \[\begin{pmatrix}
        x'\\
        y'\\
        u'\\
        v'
    \end{pmatrix} = \begin{pmatrix}
        \lambda & 0 & -\frac12(\lambda^2 - 1) & 0\\
        0 & \frac{1}{\lambda} & 0 & 0\\
        0 & 0 & \lambda & 0\\
        0 & \frac12(1 - \frac{1}{\lambda^2}) & 0 & \frac{1}{\lambda}
    \end{pmatrix} \begin{pmatrix}
        x \\
        y \\
        u \\
        v
    \end{pmatrix},\]
    where $\lambda = \frac{1+\sigma}{1-\sigma}$, is constant on $I_{\sigma^2}$. There are no periodic orbits. 
\end{prop}

\end{appendix}


\begin{thebibliography}{99}


\bibitem{albers2024outer} P.~Albers, A.~Chavez Caliz, S.~Tabachnikov. {\em Outer symplectic billiards.} arXiv preprint  arXiv:2409.07990 (2024) (To appear in Journal of Symplectic Geometry).

\bibitem{albers2026outer} P.~Albers, A.~Chavez Caliz, S.~Tabachnikov. {\em Outer symplectic billiard map at infinity.} Journal of Modern Dynamics, vol. 22, (2026): pp. 321-343.

\bibitem{EDSIV} V.I.~Arnold, A.B.~Givental, {\em Symplectic geometry}, Dynamical Systems IV: Symplectic geometry and its applications.
Vol.~4.~Springer Science \& Business Media, (2001).

\bibitem{diss-symp} L.~Baracco, O.~Bernardi, A.~Florio, A.~Nardi {\em Birkhoff attractors for dissipative symplectic billiards}, Journal of Dynamics and Differential Equations (2026): pp. 1-43.

\bibitem{dissipative-billiards} O.~Bernardi, A.~Florio, M.~Leguil {\em Birkhoff attractors of dissipative billiards}, Ergodic Theory and Dynamical Systems 45.4 (2025): pp. 898-1047.

\bibitem{ContRed} A.~Bravetti, C.~Jackman, D.~Sloan. {\em Scaling symmetries, contact reduction and Poincar\'e’s dream}. Journal of Physics A: Mathematical and Theoretical 56.43 (2023): 435203.

\bibitem{DiscHerg} A.~Bravetti, M.~Seri, M.~Vermeeren. {\em Contact variational integrators}. Journal of Physics A: Mathematical and Theoretical 52.44 (2019): 445206.

\bibitem{ContInt} A.~Bravetti, M.~Seri, F.~Zadra. {\em New directions for contact integrators}. International Conference on Geometric Science of Information. Cham: Springer International Publishing, (2021).

\bibitem{ChernSurfTry} S.S.~Chern, {\em Surface theory with Darboux and Bianchi}. Miscellanea Mathematica. Berlin, Heidelberg: Springer Berlin Heidelberg, (1991). 59-69.

\bibitem{ConOvsLagrConf} C.~Conley, V.~Ovsienko. {\em Lagrangian configurations and symplectic cross-ratios}. Mathematische Annalen 375.3 (2019): 1105-1145.

\bibitem{DarbouxII} G.~Darboux. {\em Le\c cons sur la th\'eorie g\'en\'erale des surfaces. Deuxi\`eme partie. Les congruences et les \'equations lin\'eaires aux d\'eriv\'ees partielles}. Gauthier-Villars, 1889.

\bibitem{gluck1983great} H.~Gluck, F.W.~Warner, {\em Great circle fibrations of the three-sphere.} Duke Math., J., 50.1 (1938): 107--132.

\bibitem{ivrii} V. Y. Ivrii, {\em Second term of the spectral asymptotic expansion for the {L}aplace-{B}eltrami operator on manifolds with
boundary.} Functional Analysis and its Applications, 14.2 (1980): pp. 98-106

\bibitem{pinball} R.~Markarian, E.J.~Pujals, M.~Sambarino, {\em Pinball billiards with dominated splitting.} Ergodic Theory and Dynamical Systems 30.6 (2010): pp. 1757-1786.

\bibitem{Mo1} J.~Moser. {\em Stable and random motions in dynamical systems.} Ann. of Math. Studies 77, Princeton Univ. Press, Princeton, NJ, 1973.

\bibitem{Mo2} J.~Moser. {\em Is the solar system stable?} Math. Intelligencer {\bf 1} (1978), 65--71.

\bibitem{neumann1959sharing} B.~Neumann, {\em Sharing ham and eggs}. Iota, Manchester University (1959): 14-18.

\bibitem{birkhoff-poritsky} H.~Poritsky, {\em The billiard ball problem on a table with a convex boundary–An illustrative
dynamical problem} Annals of Mathematics (1950): 446-470.

\bibitem{pottmann2001computational} H.~ Pottmann, J.~Wallner, {\em Computational line geometry}. Vol. 6. Berlin: Springer, (2001).

\bibitem{SloanDS} D.~Sloan. {\em Dynamical similarity}. Physical Review D 97.12 (2018): 123541.

\bibitem{tabachnikov1993geometry} S.~Tabachnikov, {\em Geometry of Lagrangian and Legendrian 2-web}. Differential Geometry and its Applications 3.3 (1993): 265-284.

\bibitem{tabachnikov1995dual} S.~Tabachnikov. {\em On the dual problem}. Advances in Mathematics 115.2 (1995): 221-249.

\bibitem{tabachnikovProjective} S.~Tabachnikov. {\em Introducing projective billiards}. Ergodic Theory and Dynamical Systems, 17(4), 957-976, (1997).

\bibitem{tabachnikov2003three}
S.~Tabachnikov. {\em On three-periodic trajectories of multi-dimensional dual billiards}. Algebraic \& Geometric Topology 3.2 (2003): 993-1004.

\bibitem{zhao2024mechanical} A.~Takeuchi, L.~Zhao, {\em Projective integrable mechanical billiards.} Nonlinearity 37.1 (2024): 015011.

\bibitem{WillNormForms} J.~Williamson, {\em On the algebraic problem concerning the normal forms of linear dynamical systems}. American journal of mathematics 58.1 (1936): 141-163.

\end{thebibliography}
\end{document}